\documentclass[a4,12pt]{amsart}
\usepackage{amssymb,amstext,amsmath,amscd,amsthm,amsfonts,enumerate,latexsym}
\usepackage{color}
\usepackage[dvipdfmx]{graphicx}
\usepackage[all]{xy}
\theoremstyle{plain}
\newtheorem{thm}{Theorem}[section]

\newtheorem*{thm*}{Theorem}
\newtheorem*{cor*}{Corollary}

\newtheorem{prop}[thm]{Proposition}

\newtheorem{lemma}[thm]{Lemma}
\newtheorem{lem}[thm]{Lemma}
\newtheorem{cor}[thm]{Corollary}

\newtheorem{claim}[thm]{Claim}
\newtheorem*{claim*}{Claim}

\theoremstyle{definition}
\newtheorem{defn}[thm]{Definition}

\newtheorem{ex}[thm]{Example}

\newtheorem{rem}[thm]{Remark}
\newtheorem{remark}[thm]{Remark}

\theoremstyle{remark}

\numberwithin{equation}{thm}

\def\pd{\operatorname{pd}}
\def\GCD{\operatorname{GCD}}
\def\Ext{\operatorname{Ext}}

\def\Im{\operatorname{Im}}

\def\Ker{\operatorname{Ker}}

\def\Hom{\operatorname{Hom}}
\def\RHom{\mathrm{{\bf R}Hom}}

\def\rank{\mathrm{rank}}

\def\e{\mathrm{e}}
\def\m{\mathfrak m}
\def\n{\mathfrak n}

\newcommand{\Ann}{\mathrm{Ann}}

\newcommand{\rma}{\mathrm{a}}

\newcommand{\rmc}{\mathrm{c}}

\newcommand{\rme}{\mathrm{e}}
\newcommand{\rmf}{\mathrm{f}}

\newcommand{\rmr}{\mathrm{r}}

\newcommand{\rmI}{\mathrm{I}}

\newcommand{\rmK}{\mathrm{K}}

\newcommand{\rmQ}{\mathrm{Q}}

\newcommand{\calA}{\mathcal{A}}
\newcommand{\calB}{\mathcal{B}}
\newcommand{\calC}{\mathcal{C}}

\newcommand{\calG}{\mathcal{G}}

\newcommand{\calR}{\mathcal{R}}
\newcommand{\calS}{\mathcal{S}}
\newcommand{\calT}{\mathcal{T}}

\newcommand{\calX}{\mathcal{X}}

\newcommand{\fka}{\mathfrak{a}}

\newcommand{\fkc}{\mathfrak{c}}

\newcommand{\fkm}{\mathfrak{m}}

\newcommand{\fkp}{\mathfrak{p}}
\newcommand{\fkq}{\mathfrak{q}}

\newcommand{\mapright}[1]{%
\smash{\mathop{%
\hbox to 1cm{\rightarrowfill}}\limits^{#1}}}

\newcommand{\mapleft}[1]{%
\smash{\mathop{%
\hbox to 1cm{\leftarrowfill}}\limits_{#1}}}

\def\AGL{\operatorname{AGL}}

\def\Ass{\operatorname{Ass}}

\def\grade{\mathrm{grade}}

\def\gr{\mbox{\rm gr}}
\def\ch{\operatorname{ch}}

\title[Ulrich ideals and $2$-$\AGL$ rings]{Ulrich ideals and $2$-$\AGL$ rings}

\author{Shiro Goto}
\address{Department of Mathematics, School of Science and Technology, Meiji University, 1-1-1 Higashi-mita, Tama-ku, Kawasaki 214-8571, Japan}
\email{shirogoto@gmail.com}

\author{Ryotaro Isobe}
\address{Department of Mathematics and Informatics, Graduate School of Science and Technology, Chiba University, 1-33 Yayoi-cho, Inage-ku, Chiba, 263-8522, Japan}
\email{r.isobe.math@gmail.com}

\author{Naoki Taniguchi}
\address{Global Education Center, Waseda University, 1-6-1 Nishi-Waseda, Shinjuku-ku, Tokyo 169-8050, Japan}
\email{naoki.taniguchi@aoni.waseda.jp}
\urladdr{http://www.aoni.waseda.jp/naoki.taniguchi/}


\thanks{The first author was partially supported by JSPS Grant-in-Aid for Scientific Research (C) 16K05112. The third author was partially supported by JSPS Grant-in-Aid for Young Scientists (B) 17K14176 and Waseda University Grant for Special Research Projects 2018K-444, 2018S-202.}

\begin{document}
\maketitle

\setlength{\baselineskip}{15pt}

\begin{abstract}
The notion of $2$-almost Gorenstein local ring ($2$-$\AGL$ ring for short) is a generalization of the notion of almost Gorenstein local ring from the point of view of Sally modules of canonical ideals. In this paper, for further developments of the theory, we discuss three different topics on $2$-$\AGL$ rings. The first one is to clarify the structure of minimal presentations of canonical ideals, and the second one is the study of the question of when certain fiber products, so called amalgamated duplications are $2$-$\AGL$ rings.  We also explore Ulrich ideals in $2$-$\AGL$ rings, mainly two-generated ones.  
\end{abstract}

{\footnotesize \tableofcontents}


\section{Introduction}

The series \cite{CGKM, GGHV, GMP, GMTY1, GMTY2, GMTY3, GMTY4, GRTT, GTT, GTT2} of researches are motivated and supported by the strong desire to stratify Cohen-Macaulay rings, finding new and interesting classes which naturally include that of Gorenstein rings. As is already pointed out by these works, the class of {\it almost Gorenstein local rings} (AGL rings for short) could be  a very nice candidate for such  classes. The prototype of AGL rings is found in the work \cite{BF} of V. Barucci and R. Fr\"oberg in 1997, where they introduced the notion of AGL ring for one-dimensional analytically unramified local rings, developing a beautiful theory on numerical semigroups. In 2013, the first author, N. Matsuoka, and T. T. Phuong \cite{GMP} extended the notion of AGL ring given by \cite{BF} to arbitrary one-dimensional Cohen-Macaulay local rings, by means of the first Hilbert coefficients of canonical ideals. They broadly opened up the theory in dimension one, which prepared for the higher dimensional notion of AGL ring provided  in 2015 by \cite{GTT}.  Subsequently in 2017, T. D. M. Chau, the first author, S. Kumashiro, and N. Matsuoka \cite{CGKM} defined the notion of $2$-${\rm AGL}$ ring as a possible successor of AGL rings of dimension one. To explain the motivations for the present researches, we need to remind the reader of $2$-AGL rings more precisely.

Throughout, let $(R, \m)$ be a Cohen-Macaulay local ring with $\dim R=1$, possessing the canonical module $\rmK_R$. We say that an ideal $I$ in $R$ is a {\it canonical ideal} of $R$, if $I \neq R$, and $I \cong \rmK_R$ as an $R$-module. In what follows, we assume that the ring $R$ possesses a canonical ideal, which contains a parameter ideal $Q=(a)$ of $R$ as a reduction. This assumption is automatically satisfied if $R$ has an infinite residue class field. Let $\calT= \calR(Q)=R[Qt]$ and $\calR = \calR(I)=R[It]$ be the Rees algebras of $Q$ and $I$ respectively, where $t$ denotes an indeterminate. We set $\calS_Q(I) = I\calR/I\calT$ and call it the {\it Sally module} of $I$ with respect to $Q$ (\cite{V1}). Let $\rme_i(I)~(i=0, 1)$ be the $i$-th Hilbert coefficients of $R$ with respect to $I$, that is, the integers satisfy the equality
$$
\ell_R(R/I^{n+1}) = \rme_0(I) \binom{n+1}{1} - \rme_1(I)  \ \ \mbox{for all}\ \ n \gg 0
$$
where $\ell_R(M)$ denotes, for each $R$-module $M$, the length of $M$.
We set $\rank~\calS_Q(I) = \ell_{\calT_\fkp}([\calS_Q(I)]_{\fkp})$ which is called the {\it rank} of $\calS_Q(I)$, where $\fkp = \m \calT$. We then have 
$$
\rank~\calS_Q(I) = \rme_1(I) - \left[\rme_0(I) - \ell_R(R/I)\right]
$$
(\cite[Proposition 2.2 (3)]{GNO}). Note that $\rank~\calS_Q(I)$ is an invariant of $R$, independent of the choice of canonical ideals $I$ and the reductions $Q$ of $I$ (see \cite[Theorem 2.5]{CGKM}). With this notation we have the following.

\begin{defn}(\cite[Definition 1.3]{CGKM})\label{1.1}
We say that $R$ is a {\it $2$-almost Gorenstein local ring} ($2$-${\rm AGL}$ ring for short), if $\rank~\calS_Q(I) =2$, that is, $\rme_1(I) = \rme_0(I) - \ell_R(R/I) + 2$.
\end{defn}

\noindent
Because $R$ is a non-Gorenstein AGL ring if and only if $\rank~\calS_Q(I) =1$ (\cite[Theorem 3.16]{GMP}), $2$-AGL rings could be considered to be one of the successors of AGL rings.

We set $K = a^{-1}I$ in the total ring $\rmQ(R)$ of fractions of $R$. Therefore, $K$ is a fractional ideal of $R$ such that $R \subseteq K \subseteq \overline{R}$ (here $\overline{R}$ stands for the integral closure of $R$ in $\rmQ(R)$) and $K \cong \rmK_R$, which we call a {\em canonical fractional ideal} of $R$.  We set $S =R[K]$. Hence, $S$ is a module-finite birational extension of $R$, and it is independent of the choice of $K$ (\cite[Theorem 2.5 (3)]{CGKM}). Let $\fkc = R:S$. We are now able to state the characterization of $2$-AGL rings given by \cite{CGKM}, which we shall often refer to, in the present paper.

\begin{thm}[{\cite[Theorem 1.4]{CGKM}}]\label{mainref} The following conditions are equivalent.
\begin{enumerate}[{\rm (1)}]
\item $R$ is a $2$-$\AGL$ ring. 
\item There is an exact sequence
$0 \to \calB(-1) \to \calS_Q(I) \to \calB(-1) \to 0$
of graded $\calT$-modules, where $\calB=\calT/\m \calT \ (\cong (R/\m)[t])$.
\item $K^2 = K^3$ and $\ell_R(K^2/K) = 2$.
\item $I^3=QI^2$ and $\ell_R(I^2/QI) = 2$.
\item $R$ is not a Gorenstein ring but $\ell_R(S/[K:\fkm])=1$.
\item $\ell_R(S/K)=2$.
\item $\ell_R(R/\fkc) = 2$.
\end{enumerate}
When this is the case, $\fkm{\cdot}\calS_Q(I) \ne (0)$, whence  the exact sequence given by condition $(2)$ is not split, and we have 
$$\ell_R(R/I^{n+1}) = \rme_0(I)\binom{n+1}{1}- \left(\rme_0(I) - \ell_R(R/I) +2\right)$$
for all $n \ge 1$.
\end{thm}


As is noted above, the notion of $2$-AGL ring could be considered to be one of the successors of the notion of AGL ring. However, if $2$-AGL rings claim that they are orthodox successors of AGL rings, it must be proved, showing that they really inherit several distinctive properties which AGL rings usually keep. In the present article, to certify the orthodoxy of $2$-AGL rings for the further studies, 
we investigate three topics on $2$-AGL rings, which are closely studied already for the case of AGL rings. The first topic concerns minimal presentations of canonical ideals. In Section 2, we will give a necessary and sufficient condition for a given one-dimensional Cohen-Macaulay local ring $R$ to be a $2$-AGL ring, in terms of minimal presentations of canonical fractional ideals. Our results Theorems \ref{3.2} and \ref{3.4a}  exactly correspond to those about AGL rings given by \cite[Theorem 7.8]{GTT}.

In Section 3, we investigate a generalization of so called {\it amalgamated duplications} of $R$ (\cite{marco}), including certain fiber products, and prove that $R$ is a $2$-$\AGL$ ring if and only if so is the fiber product $R \times_{R/\fkc} R$. By \cite[Theorem 4.2]{CGKM} $R$ is a $2$-$\AGL$ ring if and only if so is the trivial extension $R \ltimes \fkc$ of $\fkc$ over $R$, which corresponds to \cite[Theotem 6.5]{GMP} for the case of AGL rings.

In Sections 4 and 5, we are interested in Ulrich ideals in $2$-AGL rings. The existence of two-generated Ulrich ideals is basically a substantially strong condition for $R$, which we closely discuss in Section 4, especially in the case where $R$ is a $2$-AGL ring. Here, we should not rush, but should explain about what are Ulrich ideals. The  notion of Ulrich ideal/module dates back to the work \cite{GOTWY} in 2014, where the authors introduced the notion, generalizing that of MGMCM modules (maximally generated maximal Cohen-Macaulay modules)  (\cite{BHU}), and started the basic theory. The maximal ideal of a Cohen-Macaulay local ring with minimal multiplicity is a typical example of Ulrich ideals, and the higher syzygy modules of Ulrich ideals are Ulrich modules. In \cite{GOTWY, GOTWY2}, all the Ulrich ideals of Gorenstein local rings of finite CM-representation type and of dimension at most $2$ are determined, by means of the classification in the representation theory. On the other hand, in \cite{GTT2}, the first author, R. Takahashi, and the third author studied the structure of the complex $\RHom_R(R/I, R)$ for  Ulrich ideals $I$ in a Cohen-Macaulay local ring $R$ of arbitrary dimension, and proved that in a one-dimensional non-Gorenstein AGL ring $(R,\m)$, the only possible Ulrich ideal is the maximal ideal $\m$ (\cite[Theorem 2.14 (1)]{GTT2}). In Section 5, we study the natural question of how and what happens about $2$-AGL rings. To state our conclusion, let $\calX_R$ denote the set of Ulrich ideals in $R$. We then have the following, which we will prove in Section 5. The assertion exactly corresponds to \cite[Theorem 2.14 (1)]{GTT2}, the result of the case where $R$ is an AGL ring of dimension one.

\begin{thm}[Corollary \ref{4.3}]
Suppose that $(R,\m)$ is a $2$-${\rm AGL}$ ring with minimal multiplicity, possessing a canonical fractional ideal $K$. Then 
$$
\calX_R= \begin{cases}
\{\fkc, \m\}, &  \ \text{if} \ K/R~\text{is}~R/\fkc\text{-free},\\
\{\m\}, & \ \text{otherwise}.
\end{cases}
$$
\end{thm}

For one-dimensional Gorenstein local rings $R$ of finite CM-representation type, the list of Ulrich ideals is known by \cite{GOTWY}. The proof given by \cite{GOTWY} is based on the techniques in the representation theory of maximal Cohen-Macaulay modules. It might have some interests to give a straightforward proof, making use of the results in \cite[Section 12]{GTT} from a different point of view. In Section 6 we shall perform it as an appendix.

In what follows, unless otherwise specified, let $R$ be a one-dimensional Cohen-Macaulay local ring with maximal ideal $\fkm$. For each finitely generated $R$-module $M$, let $\mu_R(M)$ (resp. $\ell_R(M)$) denote the number of elements in a minimal system of generators of $M$ (resp. the length of $M$). We denote by $\mathrm{K}_R$ the canonical module of $R$.


\section{Minimal presentations of canonical ideals in $2$-$\AGL$ rings}

In this section, we explore the structure of minimal presentations of canonical ideals of $2$-${\rm AGL}$ rings. Before going ahead, we summarize some known results on $2$-AGL rings, which we shall often refer to throughout this paper. Let $(R,\m)$ be a Cohen-Macaulay local ring with $\dim R = 1$, admitting the canonical module $\rmK_R$. We assume that $R$ possesses a canonical fractional ideal $K$, that is an $R$-submodule of $\rmQ(R)$ such that $R \subseteq K \subseteq \overline{R}$, where $\overline{R}$ denotes the integral closure of $R$ in $\rmQ(R)$, and $K \cong \rmK_R$ as an $R$-module. Let $S =R[K]$ and set $\fkc = R : S$. We denote by $\rmr(R) =\ell_R(\Ext_R^1(R/\m, R))$ the Cohen-Macaulay type of $R$.

\begin{prop}[{\cite[Proposition 3.3]{CGKM}}]\label{2.3a}
Suppose that $R$ is a $2$-${\rm AGL}$ ring with $r = \rmr(R)$. Then the following assertions hold true.
\begin{enumerate}[${\rm (1)}$]
\item[$(1)$] $\fkc = K:S = R:K$.
\item[${\rm (2)}$] There is a minimal system $x_1, x_2, \ldots, x_n$ of generators of $\m$ such that $\fkc =(x_1^2) + (x_2, x_3, \ldots, x_n)$.
\item[$(3)$] $S/K \cong R/\fkc$ and $S/R \cong K/R \oplus R/\fkc$ as $R/\fkc$-modules.
\item[$(4)$] $K/R\cong (R/\fkc)^{\oplus\ell} \oplus (R/\m)^{\oplus m}$ as an $R/\fkc$-module for some $\ell>0$, $m \ge 0$ such that $\ell + m = r-1$. 
\item[$(5)$] $\mu_R(S) = r+1$.
\end{enumerate}
\end{prop}
\noindent
Therefore, if $R$ is a $2$-AGL ring, then $\ell_R(K/R)=2\ell + m$. Hence, $K/R$ is a free $R/\fkc$-module if and only if $\ell_R(K/R) = 2(r-1)$.

\medskip

Let us now fix the setting of this section. In what follows, we assume that $R=T/\fka$, $\m = \n/\fka$, for some regular local ring  $(T, \n)$  with $\dim T= n \ge 3$ and  an ideal $\fka$ of $T$ such that $\fka \subseteq \n^2$. Suppose that $R$ is not a Gorenstein ring. For each $a \in T$, let $\overline{a}$ denote the image of $a$ in $R$.

Firstly, suppose that $R$ is a $2$-AGL ring, and  write $\fkc = (x_1^2) + (x_2, x_3, \ldots, x_n)$ with a minimal system $x_1, x_2, \ldots, x_n$ of generators of $\m$ (see Proposition \ref{2.3a} (2)). We choose $X_i \in \n$ so that $x_i=\overline{X_i}$ in $R$, whence $\n = (X_1, X_2, \ldots, X_n)$. Let $J = (X_1^2)+(X_2, X_3, \ldots, X_n)$. We then have
$
T/J \cong R/\fkc
$, since $\ell_T(T/J) = \ell_R(R/\fkc)=2$, so that $\fka \subseteq J$ and $\fkc = J/\fka$. On the other hand, by Proposition \ref{2.3a} (4) we have $$K/R \cong (R/\fkc)^{\oplus \ell} \oplus (R/\m)^{\oplus m}$$ with $\ell >0, m \ge 0$ such that $\ell + m = \rmr(R)-1$. Hence, letting $K= R + \sum_{i=1}^\ell Rf_i + \sum_{j=1}^m R g_j$ with $f_i, g_j \in K$, we may assume that 
$$
\sum_{i=1}^{\ell}(R/\fkc){\cdot} \overline{f_i} \cong (R/\fkc)^{\oplus \ell}\ \ \text{and} \ \  \sum_{j=1}^m (R/\fkc){\cdot}\overline{g_j} \cong (R/\m)^{\oplus m},
$$
where $\overline{f_i}, \overline{g_j}$ denote the images of $f_i, g_j$ in $K/R$.   With this notation, we have the following, which corresponds to \cite[Theorem 7.8]{GTT} for AGL rings.

\begin{thm}\label{3.2}
The $T$-module $K$ has a minimal free presentation of the form
$$
F_1 \overset{\Bbb M}{\longrightarrow} F_0 \overset{\Bbb N}{\longrightarrow} K \to 0,
$$
where the matrices $\Bbb N$ and $\Bbb M$ are given by 
$$\Bbb N= [
\begin{smallmatrix}
-1 & f_1 f_2 \cdots f_{\ell} & g_1 g_2 \cdots g_m
\end{smallmatrix}]
$$
and $$\Bbb M =\left[ 
\begin{smallmatrix}
a_{11} a_{12} \cdots a_{1n} & \cdots & a_{\ell1} a_{\ell2} \cdots a_{\ell n} & b_{11} b_{12} \cdots b_{1n} & \cdots & b_{m1} b_{m2} \cdots b_{mn} & c_1 c_2 \cdots c_q \\
X^2_1 X_2 \cdots X_n & 0 & 0 & 0  & 0 & 0  & 0 \\
0 & \ddots & 0 & 0  & 0 & 0  & 0 \\
\vdots & \vdots & X^2_1 X_2 \cdots X_n & \vdots & \vdots & \vdots & \vdots \\
0 & 0 & 0 & X_1 X_2 \cdots X_n & 0 & 0 & 0\\
0 & 0 & 0 & 0  & \ddots & 0 & 0 \\
0 & 0 & 0 & 0  & 0 & X_1 X_2 \cdots X_n & 0 
\end{smallmatrix}\right]
$$
\vspace{0.5em}

\noindent
with $a_{ij} \in J$ $($$1 \le i \le \ell$, $1 \le j \le n$$)$, $b_{ik} \in J$ $($$1 \le i \le m$, $2 \le k \le n$$)$, and $q \ge 0$. The matrix $\Bbb M$ involves the information on a system of generators of $\fka$, and we have    
$$
\fka = \sum_{i=1}^{\ell} {\rm I}_2
\left(\begin{smallmatrix}
a_{i1} & a_{i2} & \cdots & a_{in} \\
X^2_1 & X_2 & \cdots & X_n
\end{smallmatrix}\right) + 
\sum_{i=1}^m {\rm I}_2
\left(\begin{smallmatrix}
b_{i1} & b_{i2} & \cdots & b_{in} \\
X_1 & X_2 & \cdots & X_n
\end{smallmatrix}\right) + (c_1, c_2, \ldots, c_q), 
$$ where $\rmI_2(\Bbb L)$ denotes, for a $2 \times n$ matrix $\Bbb L$ with entries in $T$, the ideal of $T$ generated by $2 \times 2$ minors of $\Bbb L$.
\end{thm}

\begin{proof}
Let 
$$
F_1 \overset{\Bbb A}{\longrightarrow} F_0 \overset{[
\begin{smallmatrix}
-1 & f_1 f_2 \cdots f_{\ell} & g_1 g_2 \cdots g_m
\end{smallmatrix}]}{\longrightarrow} K \longrightarrow 0
$$
be a part of a minimal $T$-free resolution of $K$ with $F_0 = T \oplus T^{\oplus \ell} \oplus T^{\oplus m}$, which gives rise to a presentation  
$$
F_1 \overset{\Bbb A'}{\longrightarrow} G_0 \overset{\Bbb N'}{\longrightarrow} K/R \longrightarrow 0
$$
of $K/R$, where $\Bbb N'=[
\begin{smallmatrix}
\bar{f_1} \bar{f_2} \cdots \bar{f_{\ell}} & \bar{g_1} \bar{g_2} \cdots \bar{g_m}
\end{smallmatrix}]$, and $\Bbb A'$ is the $(\ell + m) \times s$ matrix obtained from $\Bbb A$ by deleting the first row. On the other hand, since $K/R \cong (T/J)^{\oplus \ell}\oplus (T/\n)^{\oplus m}$, the $T$-module $K/R$ has a minimal presentation of the form
$$
G_1 = [T^{\oplus \ell} \oplus T^{\oplus m}]^{\oplus n}= T^{\oplus \ell n} \oplus T^{\oplus mn} \overset{\Bbb B}{\longrightarrow} G_0 = T^{\oplus \ell} \oplus T^{\oplus m} \overset{\Bbb N'}{\longrightarrow} K/R \longrightarrow 0,
$$
where the matrix $\Bbb B$ is given by
$$\Bbb B =\left[ 
\begin{smallmatrix}
X^2_1 X_2 \cdots X_n & 0 & 0 & 0  & 0 & 0   \\
0 & \ddots & 0 & 0  & 0 & 0   \\
\vdots & \vdots & X^2_1 X_2 \cdots X_n & \vdots & \vdots & \vdots  \\
0 & 0 & 0 & X_1 X_2 \cdots X_n & 0 & 0 \\
0 & 0 & 0 & 0  & \ddots & 0  \\
0 & 0 & 0 & 0  & 0 & X_1 X_2 \cdots X_n 
\end{smallmatrix}\right].\vspace{1em}
$$
\noindent
Therefore, comparing with two presentations of $K/R$, we get a commutative diagram
\vspace{0.3em}
{\footnotesize
$$
\xymatrix{
G_1 \ar[r]^{\Bbb B}\ar[d]^{\xi} & G_0 \ar[r]\ar[d]^\cong & K/R\ar[r]\ar[d]^\cong & 0\\
 F_1 \ar[r]^{\Bbb A'}\ar[d]^{\eta} & G_0 \ar[r]\ar[d]^\cong & K/R \ar[r] \ar[d]^\cong & 0 \\
G_1 \ar[r]^{\Bbb B} & G_0 \ar[r] & K/R \ar[r] & 0
\vspace{2em}}
$$}

\noindent
of $T$-modules, where $\eta \circ \xi$ is an isomorphism.
Hence, 
$
\Bbb A'\ Q= \left( \Bbb B \mid  O \right)
$
for some $s \times s$ invertible matrix $Q$ with entries in $T$ (here $O$ denotes the null matrix). Setting $\Bbb M = \Bbb A Q$, we get \vspace{0.5em}
$
\text{\large $\Bbb M$} 
=
\arraycolsep5pt
\left(
\begin{array}{@{\,}ccc@{\,}}
~ & * &~\\
\hline
&\multicolumn{1}{c}{\raisebox{-10pt}[0pt][0pt]{\large $\Bbb A'$}}\\
&&\\
\end{array}
\right){\text{\large $Q$}} \ 
=
\arraycolsep5pt
\left(
\begin{array}{@{\,}c|c@{\,}}
~ * ~ & ~ * ~\\
\hline
\raisebox{-10pt}[0pt][0pt]{\large $\Bbb B$}&\raisebox{-10pt}[0pt][0pt]{\large $O$}\\
&\\
\end{array}
\right)
$, whence a required minimal presentation  
$$
F_1 \overset{\Bbb M}{\longrightarrow} F_0 \overset{\Bbb N}{\longrightarrow} K \longrightarrow 0
$$ of $K$ follows.

Let us prove that $a_{ij}, b_{ij} \in J$. We set $Z_1 =X_1^2$, and $Z_i = X_i$ for each $2 \le i \le n$. Then,
$a_{ij}\cdot(-1) + Z_j\cdot f_i = 0$
for every $1 \le i \le \ell$ and $1 \le j \le n$, whence $a_{ij} \in J$, because $\fkc K \subseteq \fkc S = \fkc$ and $\fkc =J/\fka$. Since $b_{ij}\cdot(-1) + Z_j\cdot g_i = 0$ for $j \ge 2$, we have $b_{ij} \in J$.

The last assertion about the generating system of the defining ideal $\fka$ of $R$ follows from the fact that $Z_1, Z_2, \ldots, Z_n$ forms a regular sequence on $T$. We refer to \cite[Proof of Theorem 7.8]{GTT} for details. 
\end{proof}


As a consequence of Theorem \ref{3.2}, we have the following. It exactly corresponds to \cite[Corollary 7.10]{GTT} for AGL rings.

\begin{cor}
With the same notation as in Theorem $\ref{3.2}$, the following assertions hold true.
\begin{enumerate}[$(1)$]
\item Suppose that $n = 3$. Then, $\rmr(R) = 2$, $q = 0$, $\ell =1$, and $m=0$, so that $\Bbb M = \left[\begin{smallmatrix}
a_{11} & a_{12} & a_{13} \\
X^2_1 & X_2 & X_3 
\end{smallmatrix}\right]$.
\item If $R$ has minimal multiplicity, then $q = 0$.
\end{enumerate}
\end{cor}

\begin{proof}
(1) Consider the minimal $T$-free resolution
$$
0 \longrightarrow F_2 \overset{{}^t \Bbb M}{\longrightarrow} F_1 \longrightarrow F_0 \longrightarrow R \longrightarrow 0,
$$ 
where the matrix $\Bbb M$ has the form stated in Theorem \ref{3.2}. We then have  
$$
\rmr(R)+1 = \rank_TF_1 = \ell n + mn + q = 3{\cdot}[\rmr(R)-1]+q,
$$
so that $4-2{\cdot}\rmr(R)=q \ge 0$. Therefore, $\rmr(R)=2$, and $q=0$, since $R$ is not a Gorenstein ring. Thus, $\ell=1$, $m=0$, because $\ell + m = \rmr(R) -1$.

(2) Since $R$ has multiplicity $n$, we have $\rmr(R)=n-1$, while by \cite[{\sc Theorem} 1 (iii)]{S2}, 
$
n(n-2)=\ell n + mn + q.
$
Hence, $q=0$, because $\ell + m + 1 =n$.
\end{proof}



In this paper we will refer so often to examples arising from numerical semigroup rings, that let us explain here about a canonical form of generators for their canonical modules. 
Let $0 < a_1, a_2, \ldots, a_\ell \in \Bbb Z~(\ell >0)$ be positive integers such that $\mathrm{GCD}~(a_1, a_2, \ldots, a_\ell)=1$. We set $$H = \left<a_1, a_2, \ldots, a_\ell\right>=\left\{\sum_{i=1}^\ell c_ia_i \mid 0 \le c_i \in \Bbb Z~\text{for~all}~1 \le i \le \ell \right\}$$
and call it the numerical semigroup generated by the numbers $\{a_i\}_{1 \le i \le \ell}$. Let $V = k[[t]]$ be the formal power series ring over a field $k$. We set $R = k[[H]] = k[[t^{a_1}, t^{a_2}, \ldots, t^{a_\ell}]]$ in $V$ and call it the semigroup ring of $H$ over $k$. The ring  $R$ is a one-dimensional Cohen-Macaulay local domain with $\overline{R} = V$ and $\m = (t^{a_1},t^{a_2}, \ldots, t^{a_\ell} )$, the maximal ideal. 
Let $$\rmc(H) = \min \{n \in \Bbb Z \mid m \in H~\text{for~all}~m \in \Bbb Z~\text{such~that~}m \ge n\}$$ be the conductor of $H$ and set $\rmf(H) = \rmc(H) -1$. Hence, $\rmf(H) = \max ~(\Bbb Z \setminus H)$, which is called the Frobenius number of $H$. Let $$\mathrm{PF}(H) = \{n \in \Bbb Z \setminus H \mid n + a_i \in H~\text{for~all}~1 \le  i \le \ell\}$$ denote the set of pseudo-Frobenius numbers of $H$. Therefore, $\rmf(H)$ coincides with the $\rma$-invariant of the graded $k$-algebra $k[t^{a_1}, t^{a_2}, \ldots, t^{a_\ell}]$ and $\sharp \mathrm{PF}(H) = \rmr(R)$  (\cite[Example (2.1.9), Definition (3.1.4)]{GW}).  We set  $f = \rmf(H)$ and $$K = \sum_{c \in \mathrm{PF}(H)}Rt^{f-c}$$ in $V$. Then $K$ is a fractional ideal of $R$ such that $R \subseteq K \subseteq \overline{R}$ and $$K \cong \rmK_R = \sum_{c \in \mathrm{PF}(H)}Rt^{-c}$$ as an $R$-module (\cite[Example (2.1.9)]{GW}). Therefore, $K$ is a canonical fractional ideal of $R$. Notice that  $t^f \not\in K$ but $\m t^f \subseteq R$, whence $K:\m = K + Rt^f$.


Before stating the concrete example, let us explore the properties of $2$-$\AGL$ numerical semigroup rings.

\begin{prop}\label{2.4}
Suppose that $R$ is a $2$-$\AGL$ ring. Then 
$$
\displaystyle K/R = \bigoplus_{c \in {\rm PF}(H)\setminus \{f\}} R {\cdot} \overline{t^{f-c}}
$$
where $\overline{(*)}$ denotes the image in $K/R$.
\end{prop}

\begin{proof}
We set $r =\rmr(R)$, $f=c_r$ and write ${\rm PF}(H) =\{c_1, c_2, \ldots, c_r\}$. Let us consider 
$$
I = \{i \in \Lambda \mid \operatorname{Ann}_{R/\fkc} \overline{t^{f-c_i}} =(0)\}, \ \ 
J = \{i \in \Lambda \mid  \operatorname{Ann}_{R/\fkc} \overline{t^{f-c_i}} \ne (0)\}
$$
where $\fkc = R:R[K]$ and $\Lambda =\{1, 2, \cdots, r-1\}$. Notice that $I \cup J = \Lambda$ and $I \ne \emptyset$. Since $R$ is a $2$-$\AGL$ ring, there exists $b \in H$ such that $(0):_{R/\fkc}\m  = \fkm/\fkc = [(t^b) + \fkc]/\fkc$. Then, for each $i \in I$, we have
$t^b {\cdot} \overline{t^{f-c_i}} \ne 0$ and $\m {\cdot}\overline{t^{f-c_i+b}} =(0)$ in $K/R$. Hence $f+b-c_i \in {\rm PF}(H)$, and the elements $\{t^b{\cdot} \overline{t^{f-c_i}}\}_{i \in I}$ in $K/R$ are linearly independent over $k$. Therefore
$$
K/R = \sum_{i \in I} R{\cdot}\overline{t^{f-c_i}} \bigoplus \sum_{j \in J} R{\cdot}\overline{t^{f-c_j}} = \bigoplus_{i \in \Lambda} R {\cdot} \overline{t^{f-c_i}}
$$
as desired. 
\end{proof}

For the moment, suppose that $R$ is a $2$-$\AGL$ ring and we maintain the notation as in the proof of Proposition \ref{2.4}.
Choose $b = a_j$ for some $1 \le j \le \ell$. We then have $f+b-c_i \in {\rm PF}(H)$ for each $i \in I$, while $f-c_j \in {\rm PF}(H)$ for each $j \in J$ if $J \ne \emptyset$.
 By writing $I = \{c_1 < c_2 < \cdots < c_p\}~(p>0)$ and $J = \{d_1 < d_2 < \cdots <d_p\}~(q \ge 0)$, we have the following.

\begin{thm}
The following assertions hold true. 
\begin{enumerate}[$(1)$]
\item $f+b=c_i + c_{p+1 -i}$ for every $1 \le  i \le p$.
\item If $J \ne \emptyset$, then $f=d_j +d_{q+1-j}$ for every $1 \le  j \le q$.
\end{enumerate}
\end{thm}

\begin{proof}
The assertions follow from the fact that the maps
$$
\{c_i \mid i \in I\} \to \{c_i \mid i \in I\}, x \mapsto f+b-x, \quad \{c_j \mid j \in J\} \to \{c_j \mid j \in J\}, x \mapsto f-x
$$
are well-defined and bijective. 
\end{proof}

As a consequence, we get the following, which corresponds to the case where $J = \emptyset$.

\begin{cor}
Suppose that $R$ is a $2$-$\AGL$ ring. Then the following conditions are equivalent.
\begin{enumerate}[$(1)$]
\item $K/R \cong (R/\fkc)^{\oplus (r-1)}$ as an $R$-module.
\item There is an integer $1 \le j \le \ell$ such that $f + a_j = c_i + c_{r-i}$ for every $1 \le i \le r-1$.
\end{enumerate}
\end{cor}


Let us now go back to state the example of Theorem \ref{3.2}. With the notation of Theorem \ref{3.2}, we cannot expect $q=0$ in general, as we show in the following.


\begin{ex}[cf. {\cite[Example 5.5]{CGKM}}]\label{3.5}
Let $V = k[[t]]$ be the formal power series ring over a field $k$, and set $R = k[[t^5,t^7,t^9,t^{13}]]$. Hence, $R=k[[H]]$, the semigroup ring of the numerical semigroup $H=\left<5, 7, 9, 13\right>$. We then have ${\rm f}(H) = 11$ and $\mathrm{PF}(H) =\{8, 11\}$, whence $K=R + Rt^3$ and $R[K] = k[[t^3, t^5, t^7]] = R + Rt^3+Rt^6$. Therefore, $\fkc= (t^{10},t^7,t^9,t^{13})$ and $\ell_R(R/\fkc) =2$, so that by Theorem \ref{mainref}, $R$ is a $2$-AGL ring with $\rmr(R)= 2$. 
We are interested in the defining ideal $\fka$ of $R$.
Let $T=k[[X,Y,Z,W]]$ be the formal power series ring, and let $\varphi : T \to R$ be the $k$-algebra map defined by $\varphi (X) = t^5, \varphi(Y)=t^7, \varphi(Z) = t^9$, and $\varphi(W) = t^{13}$. Then, $R$ has a minimal $T$-free resolution
of the form
$$
\Bbb F : \ 0 \to T^2 \overset{\Bbb M}{\to} T^6 \overset{\Bbb N}{\to} T^5 \overset{\Bbb L}{\to} T \to R \to 0,$$
where the matrices $\Bbb M, \Bbb N,$ and $\Bbb L$ are given by
$$
{}^t\Bbb M =\left[ 
\begin{smallmatrix}
W & X^2 & XY & YZ & Y^2 - XZ & Z^2 - XW \\
X^2 & Y & Z & W & 0 & 0\\
\end{smallmatrix}\right],  
$$
$$
\Bbb N =\left[ 
\begin{smallmatrix}
-Z^2 + XW & 0 & X^2Z & -X^3 & 0 & W \\
Y^2 - XZ & -X^2Y & X^3 & 0 & -W & 0 \\
0 & 0 & W & -Z & 0 & Y \\
0 & -W & 0 & Y & -Z & -X \\
0 & Z & -Y & 0 & X & 0\\
\end{smallmatrix}\right], 
$$
$$
\Bbb L =\left[ 
\begin{smallmatrix}
Y^2-XZ & Z^2-XW & X^4-YW & X^3Y-ZW & X^2YZ-W^2\\
\end{smallmatrix}\right].
$$ 
The $T$-dual of $\Bbb F$ gives rise to the presentation
$$
T^6 \overset{{}^t\Bbb M}{\to} T^2 \to K \to 0
$$
of the canonical fractional ideal $K=R + Rt^3$,
 so that 
$$
K/R \cong T/(X^2, Y, Z, W) \cong R/{\fkc}.
$$ 
We have
$\fka =\Ker \varphi = \rmI_2
\left(\begin{smallmatrix}
W & X^2 & XY & YZ \\
X^2 & Y & Z & W \\
\end{smallmatrix}\right) + (Y^2 - XZ, Z^2 - XW)
$.
\end{ex}

We note one example of $2$-${\rm AGL}$ rings of minimal multiplicity, whence $q=0$.

\begin{ex}[cf. {\cite[Example 5.6]{CGKM}}]\label{3.6}
Let $V = k[[t]]$ be the formal power series ring over a field $k$, and set $R = k[[H]]$, where $H=\left<4, 9, 11, 14\right>$. Then, ${\rm f}(H) = 10$ and ${\mathrm PF}(H) = \{5, 7, 10\}$, whence $K=R + Rt^3+ Rt^5$ and $R[K] = k[[t^3, t^4, t^5]] = R + Rt^3 + Rt^5 + Rt^6$. Therefore, $\fkc = (t^8,t^9,t^{11},t^{14})$ and $\ell_R(R/\fkc) =2$, so that by Theorem \ref{mainref}, $R$ is a $2$-AGL ring possessing minimal multiplicity $4$ and $\rmr(R)=3$. We consider the $k$-algebra map $\varphi : T \to R$ defined by $\varphi (X) = t^4, \varphi(Y)=t^9, \varphi(Z) = t^{11}$, and $\varphi(W) = t^{14}$, where $T=k[[X,Y,Z,W]]$ denotes the formal power series ring. Then, $R$ has a minimal $T$-free resolution 
$$
\Bbb F : \ 0 \to T^3 \overset{\Bbb M}{\to} T^8 \overset{\Bbb N}{\to} T^6 \overset{\Bbb L}{\to} T \to R \to 0$$ where the matrices $\Bbb M, \Bbb N,$ and $\Bbb L$ are given by
$$
{}^t\Bbb M = \left[
\begin{smallmatrix}
-Z & -X^3 & -W & -X^2Y & Y & W & X^4 & X^2Z \\
X^2 & Y & Z & W & 0 & 0 & 0 & 0\\
0 & 0 & 0 & 0 & X & Y & Z & W
\end{smallmatrix}\right],
$$
$$
\Bbb N =\left[ 
\begin{smallmatrix}
-X^2Z & 0 & X^4 & 0 & 0 & 0 & W & -Z \\
0 & 0 & W & -Z & 0 & W & 0 & -Y \\
0 & W & 0 & -Y & -X^2Y & X^3 & 0 & 0 \\
-W & 0 & 0 & X^2 & 0 & -Z & Y & 0 \\
0 & Z & -Y & 0 & -W & 0 & 0 & X \\
Y & -X^2 & 0 & 0 & Z & 0 & -X & 0 \\
\end{smallmatrix}\right], 
$$
$$
\Bbb L =\left[ 
\begin{smallmatrix}
Y^2-XW & X^5-YZ & Z^2-X^2W & X^3Z-YW & X^4Y-ZW & X^2YZ - W^2\\
\end{smallmatrix}\right].
$$ 
Taking $T$-dual of $\Bbb F$, we have the presentation 
$$
T^8 \overset{{}^t\Bbb M}{\to} T^3 \to K \to 0
$$
of $K = R + Rt^3 + Rt^5$, so that 
$$
K/R \cong T/(X^2, Y, Z, W) \oplus T/\n \cong R/\fkc \oplus R/\fkm.
$$
Hence, $K/R$ is not $R/\fkc$-free.
We have
$
\Ker \varphi = \rmI_2
\left(\begin{smallmatrix}
-Z & -X^3 & -W & -X^2Y \\
X^2 & Y & Z & W \\
\end{smallmatrix}\right) + 
\rmI_2
\left(\begin{smallmatrix}
Y & W & X^4 & X^2Z \\
X & Y & Z & W \\
\end{smallmatrix}\right).
$
\end{ex}

We are now asking for a sufficient condition for $R = T/\fka$ to be a $2$-AGL ring in terms of the presentation of the canonical ideal. Let us maintain the setting in the preamble of this section, assuming $R$ possesses a canonical fractional ideal  $K$ of the form 
$$
K=R+ \sum_{i=1}^\ell Rf_i + \sum_{j=1}^m Rg_j
$$
where $f_i, g_j \in K$, and $\ell >0$, $m \ge 0$ with $\ell + m + 1=\rmr(R)$. 
We then have the following. We should compare it with \cite[Theorem 7.8]{GTT}.

\begin{thm}\label{3.4a}
Let $X_1, X_2, \ldots, X_n$ be a regular system of parameters of $T$ and assume  that $K$ has a presentation of the form
$$
F_1 \overset{\Bbb M}{\longrightarrow} F_0 \overset{\Bbb N}{\longrightarrow} K \to 0  \quad \quad \quad (\sharp)
$$
where $\Bbb N$ and $\Bbb M$ are matrices of the form  stated in Theorem $\ref{3.2}$, satisfying the condition that $a_{ij}, b_{pk} \in (X_1^2) + (X_2, X_3, \ldots, X_n)$ for every $1 \le i \le \ell$, $1 \le j \le n$, $1 \le p \le m$, and $2 \le k \le n$. Then $R$ is a $2$-${\rm AGL}$ ring.
\end{thm}

\begin{proof}
The presentation $(\sharp)$ gives rise to a presentation 
$$
F_1 \overset{\Bbb B}{\longrightarrow} G_0 \overset{\Bbb L}{\longrightarrow} K/R \longrightarrow 0
$$
of $K/R$, where 
$\Bbb L=\left[\begin{smallmatrix}
\bar{f_1} \bar{f_2} \cdots \bar{f_{\ell}} & \bar{g_1} \bar{g_2} \cdots \bar{g_m}\end{smallmatrix}\right]$ (here $\overline{*}$ denotes the image in $K/R$), and the matrix $\Bbb B$ has the form stated in the proof of Theorem \ref{3.2}. Hence$$
K/R \cong (T/J)^{\oplus \ell} \oplus (T/\n)^{\oplus m},
$$ 
so that $\n{\cdot}(K/R) \ne (0)$,  since $\ell > 0$. Therefore, $\fkc \subsetneq \m$. We set $J = (X_1^2)+(X_2, X_3, \ldots, X_n)$ and let $I = JR$. Then, since $a_{ik} \in J$, inside of  $K/R$ we get
$$
X_1^2{\cdot}f_i = \overline{a_{i1}} \ \ \text{and} \ \ X_k{\cdot}f_i = \overline{a_{ik}}
$$
for every $1 \le i \le \ell$ and $2 \le k \le n$. Hence, $X_1^2{\cdot}f_i, X_k{\cdot}f_i \in I$. We similarly have $X_k{\cdot}g_j \in I$ for all $1 \le j \le m$ and $2 \le k \le n$, because $b_{jk} \in J$. Moreover, $X_1^2{\cdot}g_j \in J$ for every $1 \le j \le m$. Thus, $IK \subseteq I$, whence $IS \subseteq I$, because $S = R[K]=K^q$ for $q \gg 0$. Therefore, $I \subseteq \fkc \subsetneq \m$, so that  $I = \fkc$, since $\ell_R(R/I) \le 2$. Thus, $\ell_R(R/\fkc) = 2$, and $R$ is a $2$-AGL ring by Theorem \ref{mainref}.
\end{proof}


As a consequence of Theorem \ref{3.4a}, we have the following.

\begin{cor} Let $(T,\n)$ be a regular local ring with $\dim T=n \ge 3$ and $\n = (X_1, X_2, \ldots, X_n)$.  Choose positive integers $\ell_1, \ell_2, \ldots, \ell_n >0$ so that $\ell_1\ge 2$ and set $\fka = \rmI_2
\left(\begin{smallmatrix}
X_1^2 & X_2 & \cdots & X_{n-1} & X_n  \\
X_2^{\ell_2} & X_3^{\ell_3} & \cdots & X_n^{\ell_n} & X_1^{\ell_1}\\
\end{smallmatrix}\right)$. Then $R=T/\fka$ is a $2$-$\AGL$ ring, for which $K/R$ is a free $R/\fkc$-module of rank $n-2$.
\end{cor}


\begin{proof} Since $\sqrt{\fka + (X_1)} = \n$, $\grade_T\fka = n-1$, so that $\fka$ is a perfect ideal of $T$, whence $R=T/\fka$ is a Cohen-Macaulay local ring with $\dim R=1$, and a minimal $T$-free resolution of $R$ is given by the  Eagon-Northcott complex associated to the matrix $\left(\begin{smallmatrix}
X_1^2 & X_2 & \cdots & X_{n-1} & X_n  \\
X_2^{\ell_2} & X_3^{\ell_3} & \cdots & X_n^{\ell_n} & X_1^{\ell_1}\\
\end{smallmatrix}\right)$
(\cite{EN}). We take the $T$-dual of the resolution and get the following presentation
$$
T^{\oplus n(n-2)} \overset{{\Bbb M}'}{\to} T^{\oplus (n-1)} \overset{\varepsilon}{\to} K_R \to 0
$$
of the canonical module $K_R$ of $R$, where the matrix ${\Bbb M}'$ is given by
{\scriptsize
$$
{\Bbb M}'=\left[
\begin{smallmatrix}
X_2^{\ell_2} -X_3^{\ell_3} \cdots (-1)^nX_n^{\ell_n}(-1)^{n+1}X_1^{\ell_1}
 & 0 &  &  &  &  \\
X_1^2 -X_2 \cdots (-1)^{n+1}X_n
 & X_2^{\ell_2} -X_3^{\ell_3} \cdots (-1)^nX_n^{\ell_n}(-1)^{n+1}X_1^{\ell_1}
 &    &  &  &  \\
   &   & \ddots  & \\
   &   &  &  &  X_1^2 -X_2 \cdots (-1)^{n+1}X_n
 &  X_2^{\ell_2} -X_3^{\ell_3} \cdots (-1)^nX_n^{\ell_n}(-1)^{n+1}X_1^{\ell_1}
\\   
   &   &  &  &  0  & X_1^2 -X_2 \cdots (-1)^{n+1}X_n 
\end{smallmatrix}\right].
$$
} Let $x_i$ denote, for each $1 \le i\le n$, the image of $X_i$ in $R$. Since $x_1^2x_1^{\ell_1}= x_2^{\ell_2}x_n$, $x_i x_1^{\ell_1} = x_{i+1}^{\ell_{i+1}}x_n$ for every $2 \le i \le n-1$ and $x_1$ is a parameter of $R$, we have that every $x_i$ is a non-zerodivisor in $R$. We set $y = \frac{x_2^{\ell_2}}{x_1^{2}}$, and 
$$f_i = \begin{cases}
x_{i+1}^{\ell_{i+1}} & \text{if} \ \ 1 \le i \le n-1\\
x_1^{\ell_1} & \text{if}\  \ i = n
\end{cases}
\ \ \ \ g_i = \begin{cases}
x_1^{2} & \text{if} \ \ i=1\\
x_i & \text{if}\ \  2 \le i \le n
\end{cases}
.$$ Then $f_i = g_iy$ for all $1 \le i \le n$, so that $y^{n} = \frac{\prod_{i=1}^nf_i}{\prod_{i=1}^ng_i}= x_1^{\ell_1-2}x_2^{\ell_2 - 1} \cdots x_n^{\ell_n -1} \in R$. Hence, $y \in \overline{R}$. Let $K = \sum_{i=0}^{n-2}Ry^i$. Therefore, $R \subseteq K \subseteq \overline{R}$. We will show that $K$ is a canonical fractional ideal of $R$. Indeed, because  
$
[\begin{smallmatrix}
-1& y& -y^2& \cdots (-1)^{n-1}y^{n-2}
\end{smallmatrix}]{\cdot}{\Bbb M}' = {\mathbf 0}$, the $T$-linear map $\psi : T^{\oplus (n-1)} \to K$ defined by $\psi ({\mathbf e}_i)= (-1)^{i}y^{i-1}$ for $1 \le i \le n-1$ (here $\{{\mathbf e}_i\}_{1 \le i \le n-1}$ denotes the standard basis of $T^{\oplus n-1}$) factors through $K_R$. Let $\sigma : K_R \to K$ be the $R$-linear map such that $\psi = \sigma \varepsilon$. Then, $K = \Im \sigma$, and $\sigma$ is a monomorphism. Indeed, assume that $X = \Ker \sigma \ne (0)$, and choose $\fkp \in \Ass_RX$. Then, $(\rmK_R)_\fkp \cong \rmK_{R_\fkp}$, since $\fkp \in \Ass_R \rmK_R$, while $K_\fkp \cong R_\fkp$, since $K$ is isomprphic to  some $\m$-primary ideal of $R$ (here $\m$ denotes the maximal ideal of $R$). Consequently, we get the exact sequence $$0 \to X_\fkp \to \rmK_{R_\fkp} \to R_\fkp \to 0$$ of $R_\fkp$-modules, which  forces $X_\fkp=(0)$, because $\ell_{R_\fkp}(\rmK_{R_\fkp})= \ell_{R_\fkp}(R_\fkp)$. This is a contradiction. Thus, $\rmK_R \cong K$. We identify $\rmK_R = K$ and $\varepsilon = \psi$. Then, because $(X_1^{\ell_1}, X_2^{\ell_2}, \ldots, X_n^{\ell_n}) \subseteq (X_1^{2}, X_2, \ldots, X_n)$, the matrix ${\Bbb M}'$ is transformed with elementary column operations into the following matrix
$$
\Bbb M =\left[ 
\begin{smallmatrix}
a_{11} a_{12} \cdots a_{1n} & \cdots  & \cdots &  a_{n-2, 1} a_{n-2, 2} \cdots a_{n-2, n}  \\
X^2_1 X_2 \cdots X_n & 0  & \cdots  & 0    \\
0 & \ddots &   & 0    \\
\vdots &  & \ \ddots & \vdots   \\
0 & 0  & \cdots & X_1^2 X_2 \cdots X_n  \\
\end{smallmatrix}\right]
$$
with $a_{ij} \in (X_1^{\ell_1}, X_2^{\ell_2}, \ldots, X_n^{\ell_n})$, 
so that Theorem \ref{3.4a} shows $R$ is a $2$-$\AGL$ ring. Since $K/R \cong (T/(X_1^2, X_2, \ldots, X_n))^{\oplus n-2}$, $K/R$ is a free $R/\fkc$-module of rank $n-2$. 
\end{proof}


\section{$2$-${\rm AGL}$ rings obtained by fiber products}

In this section we study the problem of when certain fiber products, or more generally, quasi-trivial extensions of one-dimensional Cohen-Macaulay local rings are $2$-${\rm AGL}$ rings.

Let $R$ be a commutative ring and $I$ an ideal of $R$. For an element $\alpha \in R$, we set $A(\alpha)=R\oplus I$ as an additive group and define the multiplication on $A(\alpha)$ by
$$
(a, x) \cdot (b, y) = (ab, ay + bx + \alpha (xy))
$$
for  $(a, x), (b, y) \in A(\alpha)$. Then, $A(\alpha)$ forms a commutative ring which we denote by $A(\alpha) = R \overset{\alpha}{\ltimes} I$, and call it {\it the quasi-trivial extension of $R$ by $I$ with respect to $\alpha$}. We consider $A(\alpha)$ to be an $R$-algebra via the homomorphism $\xi : R \to A(\alpha), ~a \mapsto (a, 0)$. Therefore, $A(\alpha)$ is a ring extension of $R$, and $A(\alpha)$ is a finitely generated $R$-module, when $I$ is a finitely generated ideal of $R$. Notice that if $\alpha = 0$, then $A(0) = R \ltimes I$ is the ordinary idealization $I$ over $R$, introduced by M. Nagata \cite[Page 2]{N}, and $[(0) \times I]^2=(0)$ in $A(0)$. If $\alpha = 1$, then $A(1)$ is called in \cite{marco} the amalgamated duplication of $R$ along $I$, and $$A(1) \cong R \times_{R/I} R, \ (a,i) \mapsto (a, a+i),$$ the fiber product of the two copies of the natural homomorphism $R \to R/I$. Hence, if $R$ is a reduced ring, then so is $A(1)$.

Let us note the following.

\begin{lemma}\label{lemma 3.1}
Let $(R,\m)$ be a $($not necessarily Noetherian$)$ local ring. Assume that $I \ne R$ or $\alpha \in \m$. Then $A(\alpha)$ is a local ring with maximal ideal $\m \times I$.
\end{lemma}

\begin{proof}
Let $(a,x) \in A(\alpha) \setminus (\m \times I)$. Then, $a+\alpha x \not\in \m$, since $a \not\in \m$ but $\alpha x \in \m$. Therefore, setting $b = a^{-1}$ and $y = -(a+\alpha x)^{-1}{\cdot}xb$, we get $(a,x)(b,y)=1$ in $A(\alpha)$. Hence, $A(\alpha)$ is a local ring, because $\m \times I$ is an ideal of $A(\alpha)$.
\end{proof}

\begin{remark}
When $I=R$, $A(-1)$ is not a local ring, even if $(R,\m)$ is a local ring. In fact, assume that $A(-1)$ is a local ring. Then, because $\m \times R$ is a maximal ideal of $A(-1)$ and $(1,1) \not\in \m \times R$, we have $(1,1)(b,y)= (1,0)$ for some $(b,y) \in A(-1)$, so that $b=1$ and $y+b + (-1){\cdot}1{\cdot}y=0$. This is absurd. 
\end{remark}

In what follows, let $(R,\m)$ be a one-dimensional Cohen-Macaulay local ring with a canonical fractional ideal $K$. We set $S=R[K]$ and $\fkc = R:S$. Let $T$ be a birational module-finite extension of $R$ (hence $R \subseteq T \subseteq \overline{R}$), and assume that $K \subseteq T$ but $R \ne T$. We set $I = R:T$. Hence, $I = K:T$ by \cite[Lemma 3.5 (1)]{GMP}, so that $K:I=T$.

\begin{prop}\label{lemma 3.5}
$T/K \cong \rmK_{R/I}$. Hence, $\ell_R(T/K)= \ell_R(R/I)$. 
\end{prop}

\begin{proof}
Take the $K$-dual of the exact sequence $0 \to I \to R \to R/I \to 0$, and consider the resulting exact sequence $0 \to K \to K:I \to \Ext_R^1(R/I, K) \to 0$. We then have $T/K \cong \Ext_R^1(R/I, K) = \rmK_{R/I}$, since $K :I=T$. Therefore, $\ell_R(T/K) = \ell_R(\rmK_{R/I})= \ell_R(R/I)$.
\end{proof}

 Let $\alpha \in R$ and set $A = R \overset{\alpha}{\ltimes} I$. Then, since $I \ne R$, $A$ is a Cohen-Macaulay local ring with $\dim A=1$ and $\n = \m \times I$, the unique maximal ideal (Lemma \ref{lemma 3.1}).  We are now interested in the question of when $A$ is a $2$-$\AGL$ ring. Notice that we have the extensions 
$$A \subseteq T \overset{\alpha}{\ltimes} T \subseteq \rmQ(R) \overset{\alpha}{\ltimes} \rmQ(R)= \rmQ(A)$$
of rings. We set $L = T \times K$ in $T \overset{\alpha}{\ltimes} T$. Hence, $L$ is an $A$-submodule of $T \overset{\alpha}{\ltimes} T$, and  $A \subseteq L \subseteq \overline{A}$.

We begin with the following, which plays a key role in this section.

\begin{prop}\label{3.1} $L \cong \rmK_A$ and $A[L] = T \overset{\alpha}{\ltimes} T$. 
\end{prop}

\begin{proof}
Since $I = K:T$, $I^\vee \cong T$ where $(-)^{\vee} = \Hom_R(-, K)$, and we have the natural isomorphism
$$\sigma: 
A^{\vee} = \Hom_R(R \oplus I, K) \overset{\cong}{\to} I^{\vee} \oplus R^\vee \overset{\cong}{\to}  T \oplus K=L
$$
of $R$-modules. Let $Z = T \oplus T$. Then, the $R$-module $Z$ becomes a $T \overset{\alpha}{\ltimes} T$-module by the following action
$$
(a, x) \rightharpoonup (b, y) = \left((a+\alpha x)b, ay + bx\right)
$$
for each $(a, x) \in T \overset{\alpha}{\ltimes} T$ and $(b, y) \in Z$. It is routine to check that $L$ which is considered inside of $Z$ is an $A$-submodule of $Z$, and that the above $R$-isomorphism $\sigma : A^\vee \to L$ is actually an $A$-isomorphism. We now consider the homomorphism $\psi: T \overset{\alpha}{\ltimes} T  \to Z$ of $T \overset{\alpha}{\ltimes} T$-modules defined by $\psi (1) = (1,0)$. Then, $\psi$ is an isomorphism, since $\psi(a,x) = (a+\alpha x, x) $ for each $(a,x) \in T \overset{\alpha}{\ltimes} T$. Notice that for each $(a,x) \in T \overset{\alpha}{\ltimes} T$, $(a,x) \in T \times K$ if and only if $\psi (a,x) \in T \times K$. Therefore, $L$ which is considered inside of $T \overset{\alpha}{\ltimes} T$
 is isomorphic to $\rmK_A$, because $L$ which is considered inside of $Z$ is isomorphic to $A^\vee  = \rmK_A$. Since $KT=T$, we have $L^n = T \overset{\alpha}{\ltimes} T$ for every $n \ge 2$. Thus, $A[L] = T \overset{\alpha}{\ltimes} T$, since $A[L]= \bigcup_{\ell \ge 1}L^\ell = L^n$ for $n \gg 0$. 
\end{proof}

Let $\rmr_R(I) = \ell_R(\Ext_R^1(R/\m,I))$ denote the Cohen-Macaulay type of the $R$-module $I$. 

\begin{cor}\label{type}
$\rmr(A(\alpha))= \mu_R(T)+\rmr(R)= \rmr_R(I) + \mu_R(K/I)$. Hence, the Cohen-Macaulay type of $A(\alpha)$ is independent of the choice of $\alpha \in R$. 
\end{cor}

\begin{proof} With the same notation as in Proposition \ref{3.1}, because $\n L = (\m \times I)(T \times K) = \m T \times \m K$, we have an $R$-isomorphism $L/\n L \cong T/\m T \oplus K/\m K$. Therefore,  since $R/\m = A/\n$, $\rmr(A) = \mu_R(T) + \rmr(R)$, which is independent of $\alpha$. Consequently, because $\sum_{f \in \Hom_R(I,K)}f(I)= (K:I)I = TI =I$ where the second equality follows from the fact that $I = K:T$,  by \cite[Theorem 3.3]{GKL} we get $\rmr(A) = \rmr_R(I) + \mu_R(K/I)$.
\end{proof}

We now come to the main result of this section.

\begin{thm}\label{3.2b} With the same notation as above, the following conditions are equivalent.
\begin{enumerate}[$(1)$]
\item The fiber product $R \times_{R/I}R$ is a $2$-$\AGL$ ring.
\item The idealization $R \ltimes I$ is a $2$-$\AGL$ ring.
\item $A(\alpha)=R \overset{\alpha}{\ltimes} I$ is a $2$-$\AGL$ ring for every $\alpha \in R$.
\item $A(\alpha)=R \overset{\alpha}{\ltimes} I$ is a $2$-$\AGL$ ring for some $\alpha \in R$.
\item $\ell_R(T/K)=2$.
\item $\ell_R(R/I)=2$.
\end{enumerate}
\end{thm}

\begin{proof}
We maintain the same notation as in Proposition \ref{3.1}. Since $A[L] = T \overset{\alpha}{\ltimes} T$, $A[L]/L\cong T/K$ as an $R$-module, so that $\ell_A(A[L]/L)= \ell_R(T/K)$, because $R/\m = A/\n$. Thus, the assertion readily follows from Proposition \ref{lemma 3.5}, Theorem \ref{mainref}, and Proposition \ref{2.3a}.
\end{proof}

\begin{cor}\label{3.3}
Suppose that $R$ is a $2$-$\AGL$ ring. If $A(\alpha)=R \overset{\alpha}{\ltimes} I$ is a $2$-$\AGL$ ring for some $\alpha \in R$, then $T=S$ and $I=\fkc$.\end{cor}

\begin{proof}
We have $S = R[K] \subseteq T$, since $K \subseteq T$. Therefore, $S=T$, because $\ell_R(T/K)=\ell_R(S/K)=2$  by Theorems \ref{mainref} and \ref{3.2b}. 
\end{proof}

Choosing $T=S$, we have the following. The equivalence of assertions (2) and (3) covers \cite[Theorem 4.2]{CGKM}. We should compare the result with \cite[Theorem 6.5]{GMP} for the assertion on AGL rings. 

\begin{cor}\label{3.4} 
Let $R$ be a one-dimensional Cohen-Macaulay local ring with a canonical fractional ideal $K$ and assume that $R$ is not a Gorenstein ring. We set $S=R[K]$ and $\fkc = R:S$. Then the following conditions are equivalent.
\begin{enumerate}[$(1)$]
\item $R \times_{R/\fkc}R$ is a $2$-$\AGL$ ring.
\item $R \ltimes \fkc$ is a $2$-$\AGL$ ring.
\item $R$ is a $2$-$\AGL$ ring.
\end{enumerate}
\end{cor}

We note one example.

\begin{ex}\label{3.5}
Let $k$ be a field and set $R=k[[t^4, t^7, t^9]]$. Then $K=R+Rt^5$, so that $R$ is an AGL ring with $\rmr(R)=2$, because $\m K \subseteq R$ (\cite[Theorem 3.11]{GMP}). Hence $\fkc = \m$.  Let $T = k[[t^4, t^5, t^6, t^7]]$. Then, $T = R + Rt^5+Rt^6$, and $I=R:T = (t^7, t^8, t^9)$. Therefore, because $\ell_R(R/I) = 2$, by Theorem \ref{3.2b} and Corollary \ref{type} $A(\alpha) = R \overset{\alpha}{\ltimes} I$ is a $2$-$\AGL$ ring with $\rmr(A(\alpha))= \mu_R(T) + \rmr(R)=3+2=5$ for every $\alpha \in R$. In particular, $R \times_{R/I} R$ and $R \ltimes I$ are $2$-$\AGL$ rings. 
\end{ex}


\section{Two-generated Ulrich ideals in $2$-${\rm AGL}$ rings}

In this section, we explore Ulrich ideals in $2$-${\rm AGL} $ rings, mainly two-generated ones. One can find in \cite{GIK}, for arbitrary Cohen-Macaulay local rings of dimension one, a beautiful and complete theory of Ulrich ideals which are not two-generated.

First of all, let us briefly recall the definition of Ulrich ideals. The notion of Ulrich ideal was given by \cite{GOTWY} in arbitrary dimension. Although we will soon restrict our attention on the one-dimensional case, let us give it for arbitrary dimension. So, let $(R, \m) $ be a Cohen-Macaulay local ring with $\dim R=d \ge 0$, and $I$ an $\m$-primary ideal of $R$. We assume that $I$ contains a parameter ideal $Q$ of $R$ as a reduction.

\begin{defn} (\cite[Definition 1.1]{GOTWY})\label{2.1}
We say that $I$ is an {\it Ulrich ideal} in $R$, if  the following conditions are satisfied.
\begin{enumerate}[$(1)$]
\item $I \ne Q$, but $I^2=QI$.
\item $I/I^2$ is a free $R/I$-module.
\end{enumerate}
\end{defn}

\noindent
In Definition \ref{2.1}, condition $(1)$ is equivalent to saying that the associated graded ring $\gr_I(R) = \bigoplus_{n\ge 0} I^n/I^{n+1}$ is a Cohen-Macaulay ring with $\rma(\gr_I(R))=1-d$, where $\rma(\gr_I(R))$ denotes the $\rma$-invariant of $\gr_I(R)$. Therefore, condition $(1)$ of Definition \ref{2.1} is independent of the choice of reductions $Q$ of $I$. When $I=\m$, condition $(2)$ is automatically satisfied, and condition $(1)$ is equivalent to saying that $R$ has minimal multiplicity greater than one.

Here let us summarize a few basic result on Ulrich ideals, which we later use in this section. To state them, we need the notion of G-dimension. For the moment, let $R$ be a Noetherian ring. A {\it totally reflexive} $R$-module is by definition a finitely generated reflexive $R$-module $G$ such that $\Ext_R^{p}(G,R) = (0)$ and $\Ext^{p}_R(\Hom_R(G,R),R)=(0)$ for all $p >0$. Note that every finitely generated free $R$-module is totally reflexive. The {\it Gorenstein dimension} (G-dimension for short) of a finitely generated $R$-module $M$, denoted by ${\rm G}$-${\rm dim}_RM$, is defined as the infimum of integers $n \ge  0$ such that there exists an exact sequence
$$0 \to G_n \to G_{n-1} \to  \cdots \to G_0 \to M \to  0$$
of $R$-modules with each $G_i$ totally reflexive.
A Noetherian local ring $R$ is called G{\it-regular}, if every totally reflexive $R$-module is free. This is equivalent to saying that the equality ${\rm G}$-${\rm dim}_RM = \pd_RM$ holds true for every finitely generated $R$-modules $M$ (\cite{greg}).

\begin{prop}[{\cite[Theorem 2.5, Corollary 2.13, Theorem 2.8]{GTT2}}]\label{Ulrich} Let $I$ be an Ulrich ideal in $R$ and set $n = \mu_R(I)$. Then the following assertions hold true.
\begin{enumerate}[${\rm (1)}$]
\item[${\rm (1)}$] $(n-d){\cdot}\rmr(R/I)= \rmr(R)$.
\item[${\rm (2)}$] Suppose that there exists an exact sequence $0 \to R \to \rmK_R \to C \to 0$ of $R$-modules where $\rmK_R$ denotes the canonical module of $R$. If $n \ge d+2$, then $\Ann_RC \subseteq I$.
\item[${\rm (3)}$] $n= d+1$ if and only if ${\rm G}$-${\rm dim}_RR/I < \infty$.
\end{enumerate}
\end{prop}

Let $I$ be an $\m$-primary ideal of $R$, containing a parameter ideal $Q$ of $R$ as a reduction. Assume that $I^2=QI$ and consider the exact sequence
$$
0 \to Q/QI \to I/I^2 \to I/Q \to 0
$$
of $R$-modules. We then have that $I/I^2$ is a free $R/I$-module if and only if so is $I/Q$. If $I^2=QI$ and $\mu_R(I) =d+1$, the latter condition is equivalent to saying that $Q:_RI =I$, that is $I$ is exactly a {\em good ideal} in the sense of \cite{GIW}. It is known by \cite{GOTWY} that when $R$ is a Gorenstein ring, every Ulrich ideal $I$ in $R$ is $(d+1)$-generated (if it exists), and $I$ is a good ideal of $R$ (see \cite[Lemma 2.3, Corollary 2.6]{GOTWY}). Similarly as good ideals, Ulrich ideals are characteristic ideals, but behave very well in their nature (\cite{GOTWY, GOTWY2}). The existence of $(d+1)$-generated Ulrich ideals gives a strong influence to the structure of $R$, which we shall confirm in this section.

We now be back to the following setting. Let $(R,\m)$ be a Cohen-Macaulay local ring with $\dim R = 1$, and let $\calX_{R}$ be the set of Ulrich ideals in $R$. In general, it is quite difficult to list up the members of $\calX_R$ (see, e.g., \cite{GOTWY}). Here, to grasp what kind of sets $\calX_R$ is, first of all we explore one example. To do this, we need the following.

\begin{lem}[cf. {\cite[Proposition 3.1]{GIK2}}]\label{lem3.1}
Let $R$ be a Gorenstein local ring with $\dim R=1$. We denote by $\calA_R$ the set of birational module-finite extensions $A$ of $R$ such that $A$ is a Gorenstein ring, and set $\calA_R^0 = \{A \in \calA_R \mid \mu_R(A)=2\}$. Then, there exist bijective correspondences 
$$
\calA_R^0 \to \calX_R, \   A \mapsto R:A, \ \ \text{and}\ \ \calX_R \to \calA_R^0, \   I \mapsto I:I.
$$
\end{lem}

\begin{proof}
Let $\calG_R$ be the set of ideals $I$ in $R$ such that $I^2=aI$ and $I = (a):_RI$ for some non-zerodivisor $a \in I$. We then have by \cite[Proposition 3.1]{GIK2} a bijective correspondence $\calG_R \to \calA_R, ~I \mapsto I:I$.
Because $\calX_R = \{I \in \calG_R \mid \mu_R(I) = 2\}$ and $I:I = a^{-1}I$ for every $I \in \calG_R$ and every reduction $(a)$ of $I$, we get $\mu_R(I:I) = \mu_R(I)$, so that $I:I \in \calA_R^0$ for each $I \in \calX_R$. Conversely, let $A \in \calA_R^0$ and write $A = I:I$ with $I \in \calG_R$. Let $(a)$ be a reduction of $I$. Then, $A = I:I = a^{-1}I$, so that $\mu_R(I) = \mu_R(A) = 2$, while $R:A=I$, because $A= R:I$ by \cite[Proposition 2.5]{GIK2} and $I = R:(R:I)$ (remember that $R$ is a Gorenstein ring). Hence, $I \in \calX_R$, and the correspondences follow.
\end{proof}

\begin{ex}\label{3.2a}
Let $k$ be a field and set $R = k[[t^n, t^{n+1}, \ldots, t^{2n-2}]]~(n \ge 3)$, where $t$ is an indeterminate. Then, $R$ is a Gorenstein ring, and 
\begin{equation*}
\calX_{R} = 
\begin{cases}
\{(t^4, t^6)\} & (n=3), \\
\{(t^4-\alpha t^5, t^6) \mid \alpha \in k \} & (n=4), \\
\ \ \ \emptyset & (n \ge 5).
\end{cases}
\end{equation*}
When $n =4$,  we have $(t^4-\alpha t^5, t^6)=(t^4-\beta t^5, t^6)$, only if $\alpha =\beta$. 
\end{ex}

\begin{proof}
Our ring $R$ is a Gorenstein ring, since the numerical semigroup $H =\left<n, n+1, \ldots, 2n-2\right>$ is symmetric (\cite{HK2}). Therefore, in order to determine the members of $\calX_R$, by Lemma \ref{lem3.1} it suffices to list  the members of $\calA_R^0$, taking $R:A$ for each $A\in \calA_R^0$. We set $V = k[[t]]$.

{(1)} {\em $($The case where $n=3$$)$} Let $A \in \calA_R^0$.  Then $R \subsetneq A \subsetneq V$, whence $t^5 \in R:\m \subseteq A$, which follows from the fact that the image of $t^5$ in $\rmQ(R)/R$ is a unique socle of $\rmQ(R)/R$ and $(0) \ne A/R \subseteq \rmQ(R)/R$. Therefore
$$
k[[t^3, t^4, t^5]] \subseteq A \subseteq k[[t^2, t^3]].
$$ Hence $A=k[[t^2, t^3]]$, because $k[[t^3, t^4, t^5]]$ is not a Gorenstein ring and $\ell_{k}( k[[t^2, t^3]]/k[[t^3, t^4, t^5]])=1$. It is direct to show $R:A = R: t^2 =(t^4, t^6)$. Hence $\calX_R=\{(t^4,t^6)\}$.

{(2)} {\em $($The case where $n=4$$)$} Let $A \in \calA_R^0$. Then, $t^7 \in R:\m \subseteq A$, and $
k[[t^4,t^5, t^6, t^7]] \subseteq A \subseteq k[[t^2, t^3]]$. We have $A \not\subseteq k[[t^3,t^4, t^5]]$. Indeed, if $A \subseteq k[[t^3,t^4, t^5]]$, then $A=k[[t^3, t^4, t^5]]$ or $A= k[[t^4,t^5, t^6, t^7]]$, because $\ell_k(k[[t^3,t^4, t^5]]/k[[t^4,t^5, t^6, t^7]])=1$. This is, however, impossible, since both $k[[t^3, t^4, t^5]]$ and $k[[t^4,t^5, t^6, t^7]]$ are not a Gorenstein rings. Hence 
$$
k[[t^4,t^5, t^6, t^7]] \subsetneq A \subseteq k[[t^2, t^3]], \ \ A \not\subseteq k[[t^3,t^4,t^5]].$$ We choose $\xi \in A$ so that $\xi \not\in k[[t^3,t^4,t^5]]$. Then, since $k[[t^4,t^5, t^6, t^7]] = k + t^4V \subseteq A$, we may assume that $\xi = t^2 + \alpha t^3$ with $\alpha \in k$. Therefore, because $$k[[t^4, t^5, t^6, t^7]] \subsetneq R[\xi]=k[[t^2+\alpha t^3, t^4, t^5, t^6]] \subseteq A \subseteq k[[t^2,t^3]]$$
and $\ell_k(k[[t^2,t^3]]/k[[t^4, t^5, t^6, t^7]]) = 2$,  we have
$
\ell_k(k[[t^2, t^3]]/R[\xi]) \le 1.
$
Hence,
$A=R[\xi]$ or $A=k[[t^2, t^3]]$, where $k[[t^2,t^3]] \not\in \calA_R^0$ since 
 $\mu_R(k[[t^2,t^3]])= 3$. Thus, $A= R[\xi]$, and we have $R:A = R:R[\xi] = R: \xi = (t^4-\alpha t^5, t^6)$. Hence, $\calX_R=\{(t^4-\alpha t^5, t^6) \mid \alpha \in k\}$, because $\calA_R^0=\{R[t^2+\alpha t^3] \mid \alpha \in k\}$.

{(3)} {\em $($The case where $n= 2q + 1$ with $q \ge 2$$)$} Assume that $\calX_R \ne \emptyset$ and choose $I \in \calX_R$. We set $A = I:I$. Then $A \in \calA_R^0$. We have 
$
t^nV \subseteq k[[t^n, t^{n+1}, \ldots, t^{2n-1}]] \subseteq A,
$
since the image of $t^{2n -1}$ in $\rmQ(R)/R$ is a unique socle of $\rmQ (R)/R$. We set 
$
\calC = A : V = t^cV~(c \ge 0),
$
the conductor of $A$. Hence, $c \le n =2q+1$, because $t^nV \subseteq A$. Let $\ell = \ell_k(V/A)$. We then have $2\ell = c$, since $A$ is a Gorenstein ring (\cite[Korollar 3.5]{HK}), so that $\ell \le q$. Let $\m_A$ denote the maximal ideal of $A$. Then, $(\m_A/\m A)^2=(0)$, since $\ell_k(A/\m A)=\ell_R(A/\m A) = \mu_R(A)=2$. We look at the chain
$$
A/{\m A} \supsetneq \m_A/\m A \supsetneq (\m_A/\m A)^2 = (0)
$$
of ideals in the ring $\overline{A}=A/\m A$, and take $\xi \in \m_A$, so that $\m_A/\m A = (\overline{\xi})$ (here $\overline{\xi}$ denotes the image of $\xi$ in $\overline{A}$). Then, $\overline{\xi} \neq 0$, but ${\overline{\xi}}^2 = 0$ in $\overline{A}$. Consequently, $
\xi^2 \in \m A \subseteq t^{n}V$ and $A = R + R\xi$,
since $A/\fkm A = k + k \overline{\xi}$. Therefore, $
2{\rm \nu}(\xi) \ge n = 2q+1,
$ so that ${\rm \nu}(\xi) \ge q+1$ (here $\nu(*)$ denotes the valuation of $V$). Thus, $
A = R + R\xi \subseteq k[[t^{q+1}, t^{q+2}, \ldots, t^{2q+1}]],
$ 
whence $A = k[[t^{q+1}, t^{q+2}, \ldots, t^{2q+1}]]$, because $\ell_k(V/k[[t^{q+1}, t^{q+2}, \ldots, t^{2q+1}]]) = q$ and $\ell_k(V/A) =\ell\le q$. 
This is, however, impossible, since $A$ is a Gorenstein ring but $k[[t^{q+1}, t^{q+2}, \ldots, t^{2q+1}]]$ is not. Thus $\calX_R = \emptyset$.

{(4)} {\em $($The case where $n = 2q$ with $q \ge 3$$)$} Assume that $\calX_R \ne \emptyset$ and choose $I \in \calX_R$. We set $A = I:I$. We then have $t^{2n-1} \in A$, considering the image of $t^{2n-1}$ in $\rmQ(R)/R$. We set $\ell = \ell_k(V/A)$ and $\calC = A : V$. Then $\calC = t^{2\ell}V$, since $A$ is a Gorenstein ring. Therefore, $\ell \le q$, because $t^n V \subseteq A$ and $n = 2q$. On the other hand, considering the chain
$$
A/{\m A} \supsetneq \m_A/\m A \supsetneq (\m_A/\m A)^2 = (0)
$$
of ideals in the ring $\overline{A}=A/{\m A}$ and taking $\xi \in \m_A$ so that $\m_A/\m A= (\overline{\xi})$, we get $\overline{\xi} \neq 0$ and ${\overline{\xi}}^2 = 0$ in $\overline{A}$. Therefore, $
\xi^2 \in \m A \subseteq t^{n}V \ \ \text{and} \ \ A = R + R\xi$, 
because $A/\fkm A = k + k \overline{\xi}$. Consequently, $2{\rm \nu}(\xi) \ge n = 2q$. Hence, ${\rm \nu}(\xi) \ge q$, so that $$A= R + R\xi\subsetneq k[[t^q, t^{q+1}, \ldots, t^{2q-1}]] \subseteq V,$$ where the strictness of the first inclusion follows from the fact that $k[[t^q, t^{q+1}, \ldots, t^{2q-1}]]$ is not a Gorenstein ring. Therefore, because $\ell_k(V/A)= \ell$ and $\ell_k(V/k[[t^q, t^{q+1}, \ldots, t^{2q-1}]]) = q-1$, we get $q-1 < \ell$, whence $\ell = q$. We set $T = k[[t^{2q}, t^{q+1}, \ldots, t^{4q-1}]]$. Then $\ell_k(A/T) = q -1$, since $\ell_k(V/T)= 2q-1$. Because 
$$
A/T \subseteq V/T = k\overline{t} + k\overline{t^2} + \cdots + k\overline{t^{2q-1}},$$
where $\overline{*}$ denotes the image in $V/T$, we obtain elements $\xi_1, \xi_2, \ldots, \xi_{q-1} \in \sum_{i=1}^{2q-1}kt^i$ so that $A = T + \sum_{i=1}^{q-1} k\xi_i$. Therefore 
$$A = T[\xi_1, \xi_2, \ldots, \xi_{q-1}] = k[[\xi_1, \xi_2, \ldots, \xi_{q-1}, t^{2q}, \ldots, t^{4q-1}]],$$
 whence $\xi_1, \xi_2, \ldots, \xi_{q-1} \in \m_A$ and $(\overline{\xi_1}, \overline{\xi_2}, \ldots, \overline{\xi}_{q-1}) \subseteq \m_A/\m A$. We now notice that if $\sum_{i=1}^{q-1} a_i\xi_i \in \m A$ with $a_i \in k$, then $\sum_{i=1}^{q-1} a_i\xi_i \in t^{2q}V$, whence $a_i=0$ for all $1\le i \le q-1 $. Therefore, $1 = \ell_k(\m_A/\m A) \ge q-1 \ge 2$. This is a required contradiction. 
 
Let us make sure of the last assertion. Suppose that $n=4$ and $(t^4-\alpha t^5, t^6) = (t^4-\beta t^5, t^6)$ where $\alpha, \beta \in k$. We write $t^4- \alpha t^5 = f (t^4- \alpha t^5) + g t^6$ for some $f, g \in R$. By setting $f = c_0 + c_1t^4+c_2t^5+c_3t^6 + \eta$, $g = d_0+ d_1t^4 + d_2t^5+d_3t^6 + \xi$ for some $c_i, d_j \in k$ and $\eta, \xi \in t^8V$, we then have $t^4-\alpha t^5 = c_0t^4 -\beta c_0t^5 + d_0 t^6 + \text{(higher  terms)}$. Hence, $c_0=1$ and $\alpha = \beta c_0 = \beta$, as desired.
\end{proof}

Let us give here simple examples of $2$-AGL rings, which contain numerous two-generated Ulrich ideals.

\begin{ex}\label{2.7}
Let $(R, \m)$ be an ${\rm AGL}$ ring with $\dim R=1$ and suppose that $R$ is not a Gorenstein ring, say $R= k[[t^3,t^4,t^5]]$, the semigroup ring of $H = \left<3,4,5\right>$ over a field $k$. Let $\alpha \in \m$ and consider the quasi-trivial extension $A = R \overset{\alpha}{\ltimes} R$ of $R$ with respect to $\alpha$ (see Section 3) Then, $A$ is a $2$-AGL ring by \cite[Theorem 3.10]{CGKM}, because $A$ is a free $R$-module with $\ell_R(A/\m A) = 2$. Let $\fkq$ be a parameter ideal of $R$ and assume that $\alpha \in \fkq$. We set $I = \fkq \times R$. Then, $I$ is an Ulrich ideal of $A$ with $\mu_A(I)=2$. Therefore, if $\alpha = 0$, then $\fkq \times R$ is an Ulrich ideal of $A$ for every parameter ideal $\fkq$ of $R$ (\cite[Example 2.2]{GOTWY}). 
\end{ex}

\begin{proof}
Let $\fkq = (a)$ and set $f = (a,0) \in I$. Then, $I^2 = fI$, since $I^2 = (a^2) \times (aR + \alpha R) = (a^2) \times aR=fI$. Note that $I/fA = [(a) \oplus R]/[(a) \oplus (a)] \cong R/(a)$ and $A/I = [R \oplus R]/[(a) \oplus R] \cong R/(a)$ as $R$-modules. We then have $\ell_A(A/I) =\ell_A(I/fA) = \ell_R(R/(a))$. Hence, $A/I \cong I/fA$ as an $A$-module, because $I/fA$ is a cyclic $A$-module. Thus, $I \in \calX_A$ with $\mu_A(I)=2$.
\end{proof}

Two-generated Ulrich ideals are totally reflexive $R$-modules (Proposition \ref{Ulrich} (3)), possessing minimal free resolutions of a very restricted form. Let us note the following, which we need to prove Theorem \ref{2.3}. We include a brief proof for the sake of completeness.

\begin{prop}[{cf. \cite[Example 7.3]{GOTWY}}]\label{2.2}
Suppose that $I$ is an Ulrich ideal of $R$ and assume that $\mu_R(I)=2$. We write $I=(a, b)$ with $(a)$ a reduction of $I$. Therefore, $b^2 = ac$ for some $c \in I$. With this notation, the minimal free resolution of $I$ is given by 
$$
\Bbb F : \ \ \cdots \to R^2 \overset{
\begin{pmatrix}
-b & -c\\
a & b
\end{pmatrix}}{\longrightarrow}
R^2 \overset{
\begin{pmatrix}
-b & -c\\
a & b
\end{pmatrix}}{\longrightarrow}R^2 \overset{\begin{pmatrix}
a & b
\end{pmatrix}}{\longrightarrow} I \to 0,
$$
Hence $\pd_RI= \infty$. The ideal $I$ is so called a totally reflexive $R$-module, because $I$ is reflexive, $\Ext_R^p(I, R) =(0)$, and $\Ext_R^p(\Hom_R(I,R), R) = (0)$ for all $p >0$.
\end{prop}

\begin{proof}
Here we don't assume that $R$ is a Gorenstein ring, but the proof given in {\cite[Example 7.3]{GOTWY}} still works to get the minimal free resolution $\Bbb F$ of $I$. Since $$I= (a):_RI = (a) : I = a(R:I),$$ we have  $I \cong R:I \cong I^*$, where $I^* = \Hom_R(I,R)$.  Note that the $R$-dual $\Bbb F^*$ of $\Bbb F$ remains exact. In fact, assume that $\left(\begin{smallmatrix}
-b&a\\
-c&b
\end{smallmatrix}\right)\left(\begin{smallmatrix}
x\\
y
\end{smallmatrix}\right)=0$. Then, since $-bx + ay = 0$, we have $\left(\begin{smallmatrix}
-y\\
x
\end{smallmatrix}\right)=\left(\begin{smallmatrix}
-b&-c\\
a&b
\end{smallmatrix}\right)\left(\begin{smallmatrix}
f\\
g
\end{smallmatrix}\right)$ for some $f,g \in R$. Therefore, $\left(\begin{smallmatrix}
x\\
y
\end{smallmatrix}\right)=\left(\begin{smallmatrix}
-b&a\\
-c&b
\end{smallmatrix}\right)\left(\begin{smallmatrix}
-g\\
f
\end{smallmatrix}\right)$, which shows that $\Bbb F^*$ is exact, because $\left(\begin{smallmatrix}
-b&a\\
-c&b
\end{smallmatrix}\right)^2=0$. Consequently, $\Ext_R^p(I,R)=(0)$ for all $p > 0$, whence $\Ext_R^p(I^*,R)=(0)$ for all $p >0$, because $I \cong I^*$. On the other hand, by the above argument we have an exact sequence $$0 \to I \to R^{\oplus 2} \to I \to 0$$ whose $R$-dual $0 \to I^* \to  R^{\oplus 2} \to I^* \to 0$ remains exact. Therefore, $I$ is a reflexive $R$-module. Thus, $I$ is a totally reflexive $R$-module.
\end{proof}

We now start the analysis of the question of how many two-generated Ulrich ideals are contained in a given $2$-AGL ring. Let $K$ be a  canonical fractional ideal of $R$. Let $S =R[K]$ and set $\fkc = R : S$. We then have the following, which shows the existence of two-generated Ulrich ideals is a substantially strong restriction.

\begin{thm}\label{2.3}
Suppose that $R$ is a $2$-${\rm AGL}$ ring and let $K$ be a canonical fractional ideal of $R$. Let $\fkc = (x_1^2) + (x_2, x_3, \ldots, x_n)$ with a minimal system $x_1, x_2, \ldots, x_n$ of generators of $\m$. Assume that $R$ contains an Ulrich ideal $I$ with $\mu_R(I)=2$. Then the following assertions hold true.
\begin{enumerate}[$(1)$]
\item $K/R$ is a free $R/\fkc$-module. 
\item $I + \fkc = \m$.
\item $\fkc = (x_2, x_3, \ldots, x_n)$.
\end{enumerate}
Consequently, $\mu_R(\fkc) =n-1$, and $x_1^2 \in (x_2, x_3, \ldots, x_n)$.
\end{thm}

\begin{proof}
(1) We have $K/R \cong (R/\fkc)^{\oplus \ell} \oplus (R/\m)^{\oplus m}$ with $\ell > 0, m \ge 0$ such that $\ell +m = \rmr(R)+1$ (Proposition \ref{2.3a} (4)). To show assertion $(1)$, let us assume that $m>0$. Then, since $I$ is totally reflexive (Proposition \ref{2.2}) and $\Ext_R^p(I,K)= (0)$ for every $p > 0$, we get $\Ext_R^p(I, K/R)=(0)$, so that
$$
\Ext_R^p(I, R/\m)=(0) \ \ \text{for all} \ \ p>0,
$$
because $R/\m$ is a direct summand of $K/R$. This is impossible, since $\pd_RI=\infty$. Hence, $m=0$, and $K/R$ is $R/\fkc$-free.

(2) Let us use the same notation as in Proposition \ref{2.2}. Hence, $I$ has a minimal free resolution of the form
$$
\Bbb F : \ \ \cdots \to R^2 \overset{
\begin{pmatrix}
-b & -c\\
a & b
\end{pmatrix}}{\longrightarrow}
R^2 \overset{
\begin{pmatrix}
-b & -c\\
a & b
\end{pmatrix}}{\longrightarrow}R^2 \overset{\begin{pmatrix}
a & b
\end{pmatrix}}{\longrightarrow} I \to 0.
$$ Remember that $\Ext_R^p(I, R/\fkc) = (0)$ for all $p >0$, because $\Ext_R^p(I,K)= \Ext^p_R(I,R)= (0)$ for all $p >0$ and $R/\fkc$ is a direct summand of $K/R$. Let $\overline{x}$ denote, for each $x \in R$, the image of $x$ in $R/\fkc$. Then, taking the $R/\fkc$-dual of the resolution $\Bbb F$, we get the exact sequence  
$$ (E)\ \ \ 
0 \to \Hom_R(I, R/\fkc) \to (R/\fkc)^{\oplus 2} \overset{
\begin{pmatrix}
-\overline{b} & \overline{a}\\
-\overline{c} & \overline{b}
\end{pmatrix}}{\longrightarrow}
(R/\fkc)^{\oplus 2} \overset{
\begin{pmatrix}
-\overline{b} & \overline{a}\\
-\overline{c} & \overline{b}
\end{pmatrix}}{\longrightarrow}
(R/\fkc)^{\oplus 2}
 \to \cdots, 
$$
which shows that $I \nsubseteq \fkc$. Therefore,  $I+\fkc =\m$, since $\ell_R(R/\fkc) = 2$.

(3) To show assertion $(3)$, we notice that $\m/\fkc = (\overline{x_1}) = (\overline{a}, \overline{b})$, and $\ell_R(\m/\fkc)=1$. Hence $\m^2 \subseteq \fkc$. We set $J = (x_2, x_3, \ldots, x_n)$, and consider the following two cases.

{Case 1 ($\overline{a} \ne 0$).}  
Let us write $a = \alpha x_1 + \xi$ for some $\alpha \in R$ and $\xi \in J$. Then, since $\overline{a} \ne 0$, $\alpha \notin \m$ and $\overline{b} \in \fkm/\fkc = (\overline{a})$. Let $b = \beta a + \gamma$ with $\beta \in R$ and $\gamma \in \fkc$. Then, $I=(a,b) = (a, \gamma)$, whence replacing $b$ with $\gamma$, we may assume that $\alpha=1$ and $b \in \fkc$. Therefore,
$
\left(\begin{smallmatrix}
-\overline{b} & \overline{a}\\
-\overline{c} & \overline{b}
\end{smallmatrix}\right)=
\left(\begin{smallmatrix}
0 & \overline{x_1}\\
-\overline{c} & 0
\end{smallmatrix}\right),
$
so that $\overline{c} \ne 0$ by the exactness of the sequence ($E$). Consequently, writing $c= \delta x_1 + \rho$ with $\delta \notin \m$ and $\rho \in J$, we have $
\delta x_1^2 \equiv ac = b^2 \equiv 0 \ \text{mod} \ J.
$ 
Hence $x_1^2 \in J$, so that $\fkc =J$, as claimed.

{Case 2 ($\overline{a} = 0$).} Let $a = \alpha x_1^2 + \beta$ with $\alpha \in R$ and $\beta \in J$. Let us write $b = \gamma x_1 + \delta$ with $\gamma \in R$ and $\delta \in J$. Then, since $\m/\fkc = (\overline{b}) \ne (0)$, we get $\gamma \not\in \m$. Let $c= \rho x_1 + \eta$ with $\rho \in R$ and $\eta \in J$. Then, since $b^2 = ac$, we have
$
\gamma^2 x_1^2 \equiv \alpha \rho x_1^3 \ \text{mod}  \ J.
$
Hence, $x_1^2 \in J$, and  $\fkc = J$.
\end{proof}

\begin{cor}
Suppose that $(R,\m)$ is a $2$-${\rm AGL}$ ring with infinite residue class field. Let $I$ be an Ulrich ideal $I$ in $R$ with $\mu_R(I)=2$.  Then, there exists a minimal system $x_1, x_2, \ldots, x_n$ of generators of $\m$ and $b \in \fkc$ such that $\fkc = (x_2, x_3, \ldots, x_n)$ and $I = (x_1, b)$ with $I^2 = x_1I$.
\end{cor}

\begin{proof} Choose a minimal system  $x_1, x_2, \ldots, x_n$ of generators of $\m$ such that $\fkc = (x_1^2) + (x_2, x_3, \ldots, x_n)$. Then, $\fkc = (x_2, x_3, \ldots, x_n)$ by Theorem \ref{2.3}. We write $I = (a,b)$ where both the ideals $(a), (b)$ are reductions of $I$. If $a \not\in \fkc$, then since $\m/\fkc = (\overline{a})$ where $\overline{*}$ denotes the image in $\m/\fkc$, we get $b = \alpha a + \beta$ with $\alpha \in R$ and $\beta \in \fkc$. Hence $I = (a,b) = (a, \beta)$ and $\m = (a) + \fkc$. If $a \in \fkc$, then $\m/\fkc = (\overline{b})$, so that $a = \alpha b + \beta$, whence $I = (\beta, b)$. 
\end{proof}

The following result is involved in \cite[Theorem 2.8]{GTT2}, since Cohen-Macaulay local rings of minimal multiplicity are G-regular (\cite{greg}). Let us give a brief proof in our context.

\begin{cor}\label{1.7}
Let $R$ be a $2$-${\rm AGL}$ ring and let $K$ be a canonical fractional ideal of $R$. Assume that $S=R[K]$ is a Gorenstein ring. If $R$ has minimal multiplicity, then $R$ contains no two-generated Ulrich ideals.
\end{cor}

\begin{proof}
Because $R=K:K$ (\cite[Bemerkung 2.5]{HK}) and $KS = S$, 
$$\fkc = R : S= K :S \cong \Hom_R(S,K).$$ Therefore, since $S$ is a Gorenstein ring, we have $\fkc \cong S$, so that 
$n-1 = \mu_R(\fkc) = \mu_R(S) = \rmr(R) +1,
$ where the first (resp. third) equality follows from Theorem \ref{2.3} (resp. Proposition \ref{2.3a} (4)). Thus, $R$ doesn't have minimal multiplicity, because $\rmr(R) = n-1$ otherwise.
\end{proof}


The condition that $\fkc \in \calX_R$ is a strong restriction on $2$-${\rm AGL}$ rings $R$. We need the following, in order to see that $2$-AGL rings might contain Ulrich ideals, which are not two-generated.

\begin{prop}\label{1.6}
Suppose that $R$ is a $2$-${\rm AGL}$ ring, possessing a canonical fractional ideal $K$. Let $S = R[K]$ and set $\fkc = R:S$. Then the following conditions are equivalent.
\begin{enumerate}[${\rm (1)}$] 
\item[${\rm (1)}$] $\fkc \in \calX_R$.
\item[${\rm (2)}$] $S$ is a Gorenstein ring and $K/R$ is a free $R/\fkc$-module.\end{enumerate}
\end{prop}

\begin{proof}
(2) $\Rightarrow$ (1) Since $\fkc = K:S \cong \Hom_R(S,K)$, we have $\fkc = fS$ for some $f \in S$, whence $\fkc^2 = f \fkc$. Therefore, $\fkc/\fkc^2$ is a free $R/\fkc$-module if and only if so is $S/R$, because $\fkc/fR \cong S/R$. The latter condition is equivalent to saying that $K/R$ is a free $R/\fkc$-module, which follows from the exact sequence 
$$
0 \to R/\fkc \to S/\fkc \to S/R \to 0
$$
and the fact that $S/R \cong K/R \oplus R/\fkc$ (Proposition \ref{2.3a} (3)). 

(1) $\Rightarrow$ (2) By \cite[Corollary 3.8]{GMP}, $S$ is a Gorenstein ring, since $\fkc^2 = f \fkc$ for some $f \in \fkc$. Therefore, $\fkc = fS$ for some $f \in \fkc$, since $\fkc = K:S$. Thus, $\fkc/\fkc^2 \cong S/fS = S/\fkc$, whence $S/\fkc$ is a free $R/\fkc$-module. Consequently, $K/R$ is a free $R/\fkc$-module, since  $S/R \cong K/R \oplus R/\fkc$.
\end{proof}

Let us explore an example, which shows the set $\calX_R$ depends on the characteristic of the base fields. For the ring stated in Example \ref{difficult}, we have the complete list of Ulrich ideals in it.


\begin{ex}\label{difficult}
Let $V = k[[t]]$ be the formal power series ring over a field $k$ and set $R = k[[t^6,t^8,t^{10},t^{11}]]$. Then the following assertions hold true. 
\begin{enumerate}[${\rm (1)}$]
\item[${\rm (1)}$] $R$ is a $2$-AGL ring with $\rmr(R) = 2$ and $S= k[[t^2, t^{11}]]$ is a Gorenstein ring with $\fkc = (t^6, t^8, t^{10}) \in \calX_R$. We have $\fkc = (x_2,x_3,x_4)$ and $x_1^2 \in \fkc$, where $x_1 =t^{11}, x_2=t^6, x_3=t^8$, and $x_4=t^{10}$.
\item[${\rm (2)}$] Let $I \in \calX_R$ and set $n = \mu_R(I)$. Then, $n =2, 3$, and $n=3$ if and only if $I = \fkc$.
\item[${\rm (3)}$] If $\ch k \ne 2$, then the set of two-generated Ulrich ideals is 
$$
\left\{(t^6 + c_1 t^8 + c_2 t^{10}, t^{11}) \mid c_1, c_2 \in k\right\} \cup \left\{(t^8+c_1t^{10} + c_2 t^{12}, t^{11})\mid c_1, c_2 \in k\right\}
$$
and we have the following. 
\begin{enumerate}[$(i)$]
\item $(t^6 + c_1 t^8 + c_2 t^{10}, t^{11})=(t^6 + d_1 t^8 + d_2 t^{10}, t^{11})$, only if $c_1 =d_1$ and $c_2 =d_2$. 
\item $(t^8+c_1t^{10} + c_2 t^{12}, t^{11}) = (t^8+d_1t^{10} + d_2 t^{12}, t^{11})$, only if $c_1 =d_1$ and $c_2 =d_2$.
\end{enumerate}
\item[${\rm (4)}$] If $\ch k = 2$, then the set of two-generated Ulrich ideals is 
\begin{eqnarray*}
&&\left\{(t^6 + c_1 t^8 + c_2 t^{10}, t^{11}) \mid c_1, c_2 \in k\right\} \cup \left\{(t^8+c_1t^{10} + c_2 t^{12}, t^{11} +d t^{12})\mid c_1, c_2, d \in k\right\} \\
&& \cup \left\{(t^6 + c_1t^8+c_2t^{11}, t^{10} + d t^{11})\mid c_1, c_2, d \in k, d \ne 0\right\}
\end{eqnarray*}
and we have the following. 
\begin{enumerate}[$(i)$]
\item $(t^6 + c_1 t^8 + c_2 t^{10}, t^{11})=(t^6 + d_1 t^8 + d_2 t^{10}, t^{11})$, only if $c_1 =d_1$ and $c_2 =d_2$. 
\item $(t^8+c_1t^{10} + c_2 t^{12}, t^{11} +d t^{12}) = (t^8+d_1t^{10} + d_2 t^{12}, t^{11} +e t^{12})$, only if $c_1 =d_1$, $c_2 =d_2$, and $d=e$.
\item $(t^6 + c_1t^8+c_2t^{11}, t^{10} + d t^{11})= (t^6 + d_1t^8+d_2t^{11}, t^{10} + e t^{11})$, only if $c_1 =d_1$, $c_2 =d_2$, and $d=e$.
\end{enumerate}
\item[${\rm (5)}$] The Ulrich ideals in $R$ generated by monomials in $t$ are $\left\{(t^6, t^{11}), (t^8, t^{11}), \fkc= (t^6,t^{8},t^{10})\right\}$.  
\end{enumerate}
\end{ex}

\begin{proof}
(1) Because $K = R + Rt^2$, we have $\rmr(R)=2$ and $S= k[[t^2, t^{11}]]$, so that $S$ is a Gorenstein ring, and $\fkc = (t^6,t^{8},t^{10})$, since $S = R + Rt^2 + Rt^4$. We have $\ell_R(S/K) = 2$, since $S/K = k{\cdot}\overline{t^4} + k{\cdot}\overline{t^{15}}$, where $\overline{t^4}$ and $\overline{t^{15}}$  denote the images of $t^4$ and $t^{15}$ in $S/K$, respectively. Therefore, $R$ is a $2$-AGL ring by Theorem \ref{mainref}. Because $K/R$ is a cyclic $R$-module, $K/R \cong R/\fkc$, whence $\fkc \in \calX_R$ by Proposition \ref{1.6}.

(2) Since $(n-1){\cdot}\rmr(R/I) = \rmr(R)=2$ by Proposition \ref{Ulrich} (1), we get $n = 2,3$. Suppose $n = 3$. Then, $\fkc \subseteq I$ by Proposition \ref{Ulrich} (2), since $K/R \cong R/\fkc$. On the other hand, if $\fkc \subsetneq I$, we then have by \cite[Theorem 3.1]{GIK} $\fkc = bcS$ for some $b, c \in \m$. This is, however, impossible, because $\fkc = t^6S$ and $b,c \in \m \subseteq t^6V$. Therefore,  $I= \fkc$, if $n=3$.

(3), (4) We denote by $\nu(*)$ the valuation of $V$. Let $I \in \calX_R$ and suppose that $\mu_R(I) =2$. Let us write $I=(a, b)$ where $a, b \in R$. First we may assume $I^2 = a I$ and $\nu(a) < \nu(b)$. We then have $\nu(a) < 11$. Indeed, if $\nu(a) \ge 12$, then  $a, b \in \fkc=(t^6, t^8, t^{10})$, so that $I \subseteq \fkc$, which is absurd (remember that $I + \fkc = \m$).  
Besides, we notice that $\nu(a)$ is even, because $I/(a) \cong R/I$ as an $R$-module. Therefore, $\nu(a) = 6, 8, 10$. In addition, we have the following.

\begin{claim}\label{4.12}
The following assertions hold true. 
\begin{enumerate}[$(i)$]
\item If $\nu(a) = 6$, then $\nu(b) = 10, 11$. 
\item If $\nu(a) = 8$, then $\nu(b) = 11$.
\item One has $\nu(a) \ne 10$.
\item If $\ch k \ne 2$, then $(\nu(a), \nu(b)) \ne (6, 10)$.
\end{enumerate}
\end{claim}
 
\begin{proof}[Proof of Claim \ref{4.12}]
$(i)$ We first consider the case where $\nu(a) = 6$. Then we get $\nu(b) <12$. In fact, if $\nu(b) \ge 12$, then the images of $1, t^8, t^{10}, t^{11}$ in $R/I$ are linearly independent over the field $k$, so that $\ell_R(R/I) \ge 4$. This makes a contradiction, because $I/(a) \cong R/I$. Hence $\nu(b) \le 11$. We are now assuming that $\nu(b) = 8$. Since $b^2 = ac$ for some $c \in I$, we notice that $\nu(c) = 10$. Let us write $c = a \rho + b \eta$ where $\rho, \xi \in \m$. We then have $c \in t^{12}V$, which is impossible. Consequently, $\nu(b) = 10, 11$ as claimed.

$(ii)$ Suppose that $\nu(a) = 8$. By setting $b^2 = a^2 \varphi + ab \psi$ for some $\varphi, \psi \in \m$, we have $\nu(b) \ne 10$.  Let us write $a = t^8 + \alpha t^{10} + \beta t^{11} + \xi$, where $\alpha, \beta \in k$ and $\xi \in R$ with $\nu(\xi) \ge 12$. If $\nu(b) \ge 12$, then $b \in \fkc$, so that $a \notin \fkc$, because $I+\fkc = \fkm$. Hence $\beta \ne 0$ (remember that $\fkc=(t^6, t^8, t^{10})$). Therefore, if $\nu(b) \ge 14$ (resp. $\nu(b) = 12$), then the images of $1, t^6, t^8, t^{10}, t^{12}$ (resp. $1, t^6, t^8, t^{10}, t^{14}$) in $R/I$ are linearly independent over $k$, so that $\ell_R(R/I) \ge 5$, which makes a contradiction, because $R/I \cong I/(a)$. Hence $\nu(b) = 11$. 

$(iii)$ Let us assume that $\nu(a) = 10$. Since $b^2 \in (a^2, ab)$, we have $\nu(b) \ne 11, 12$, whence $\nu(b) \ge 14$. Thus $b \in \fkc$ and $a \notin \fkc$. Then the images of $1, t^6, t^8, t^{10}, t^{14}, t^{16}$ in $R/I$ are linearly independent over $k$, which is absurd.

$(iv)$ Suppose that $\nu(a) = 6$ and $\nu(a) = 10$. We may assume $a = t^6 + c_1t^8 + c_2t^{11} + c_3 t^{19}$ and $b= t^{10} + d_1t^{11} + d_2 t^{19}$, where $c_i, d_ j \in k$. Look at the isomorphism $R/I \cong k[Y, W]/\fka$, where $\fka$ is the ideal of $k[Y, W]$ generated by 

{\small
$$
(-c_1Y -c_2 W-c_3YW)^3 - Y(-d_1 W -d_2YW), Y^2 -(-c_1 Y -c_2 W -c_3 YW)(-d_1W -d_2YW),
$$
$$
(-d_1W - d_2 YW)^2 - (-c_1Y -c_2 W-c_3YW)^2Y, \ \ \text{and} \ \ W^2 -(-c_1Y -c_2 W-c_3YW)^2(-d_1W - d_2 YW).
$$}
\vspace{-1em}

\noindent
Hence $(Y, W)^3 + \fka = (Y, W)^3 + (Y^2, d_1YW, W^2)$. If $d_1 = 0$, then $\ell_R(R/I) \ge 4$, which is impossible. Therefore $d_1 \ne 0$. Since $I^2 = a I$, we can write $b^2 = a^2 \varphi + ab \psi$ for some $\varphi, \psi \in \m$.  By comparing the coefficients of $t^{21}$, we have $2d_1 = 0$, so that $\ch k =2$. Consequently, if $\ch k \ne 2$, then $(\nu(a), \nu(b)) \ne (6, 10)$, as desired.
\end{proof}

Notice that, for each $0 \ne f \in R$,  we have $t^{n+16}V \subseteq (f)$, where $n =\nu(f)$. It follows from the equalities $t^{n+16}V = fV {\cdot} t^{16}V = f{\cdot} (R:V)$ and the fact that $(R:V)$ is an ideal of $R$.

(3) Suppose that $\ch k \ne 2$. First we consider the case where $\nu(a) = 6$ and $\nu(b) = 11$.  Then $t^{33}V \subseteq (ab)$, so that $I = (t^6 + c_1 t^8 + c_2 t^{10}, t^{11})$ for some $c_1, c_2 \in k$. On the other hand, if we set $J = (t^6 + c_1 t^8 + c_2 t^{10}, t^{11})$ with $c_1, c_2 \in k$, then $J$ is an Ulrich ideal of $R$.  Let $a = t^6 + c_1 t^8 + c_2 t^{10}$. Notice that $t^n \in a J$ for each even integer $n \ge 18$, because $t^n = t^{n-12}{\cdot}a^2 - t^{n-12}{\cdot} (c_1^2t^{16} + \cdots + c_2^2 t^{20})$. Therefore, $J^2 = a J + (t^{22}) = a J$. Moreover, we have the isomorphism $R/J \cong k[Y, Z]/\fka$, where 
$$
\fka= \left((-c_1 Y-c_2Z)^3 -YZ, Y^2 - (-c_1 Y-c_2Z)Z, Z^2 -(-c_1 Y-c_2Z)^2Y, (-c_1 Y-c_2Z)Z\right)
$$ 
which yields $\ell_R(R/J) = 3$, because $\fka + (Y, Z)^3 = (Y, Z)^2$. Hence $R/J \cong J/(a)$, so that $J \in \calX_R$.

Let us assume $\nu(a) = 8$ and $\nu(b) = 11$. We may assume $a = t^8 + c_1t^{10} + c_2t^{12}$ and $b = t^{11} + dt^{12}$ where $c_1, c_2, d \in k$. The equality $I^2 = aI$ yields that $2d=0$ by comparing the coefficients of $t^{23}$. Hence $d=0$. Conversely, let $J = (t^8 + c_1t^{10} + c_2t^{12}, t^{11})$ for some $c_1, c_2 \in k$. Then $t^n \in a J$ for each even integer $n \ge 22$, where $a = t^8 + c_1t^{10} + c_2t^{12}$. We have the isomorphism $R/J \cong k[X, Z]/\fka$, where 
$$
\fka= \left(X^3 - (-c_1 Z-c_2X^2)Z, (-c_1 Z-c_2X^2)^2 - XZ, Z^2 -X^2(-c_1 Z-c_2X^2), -X^2Z\right)
$$ 
while $\fka = (X^3, Z^2, X^2Z, XZ) = (X^3, XZ, Z^2)$. Therefore, $\ell_R(R/J)=4$ and $J \in \calX_R$. The last assertions follow from the same technique as in the proof of Example \ref{3.2a}.

(4) Suppose that $\ch k = 2$. Thanks to the proof of (3), if $\nu(a) = 6, \nu(b) = 11$ (resp. $\nu(a) = 8, \nu(b) = 11$), then we have $I = (t^6 + c_1 t^8 + c_2 t^{10}, t^{11})$ (resp. $I = (t^8 + c_1 t^{10} + c_2 t^{12}, t^{11} + dt^{12})$) where $c_1, c_2 \in k$ (resp. $c_1, c_2, d \in k$). 

Let us assume $\nu(a) = 6$ and $\nu(b) = 10$. We then have $I = (t^6 + c_1 t^8 + c_2 t^{11} + c_3 t^{19}, t^{10} + d_1t ^{11} + d_2t^{19})$ for some $c_1, c_2, c_3, d_1, d_2 \in k$.
Consider the same isomorphism $R/I \cong k[Y, W]/\fka$ as in the proof of Claim \ref{4.12} $(iv)$. Then $(Y, W)^3 + \fka = (Y, W)^3 + (Y^2, d_1YW, W^2)$. Since $I \in \calX_R$, we have $\ell_R(R/I) = 3$, whence $d_1 \ne 0$ and $\fka = (Y, W)^2$.  Therefore, $YW \in \fka$ and $t^{19} \in I$. Consequently, $I = (t^6 + c_1 t^8 + c_2 t^{11}, t^{10} + d_1t ^{11})$. For the converse, let $J = (t^6 + c_1 t^8 + c_2 t^{11}, t^{10} + d_1t ^{11})$ and set $a=t^6 + c_1 t^8 + c_2 t^{11}$, where $c_1, c_2, d_1 \in k$ and $d_1 \ne 0$. Since $d_1 \ne 0$, we see that $\ell_R(R/J) = 3$ by the above isomorphism $R/J \cong k[Y, W]/\fka$. The fact that $t^n \in (a^2)$ for each even integer $n \ge 20$ implies $(t^{10} + d_1t ^{11})^2 \in (a^2)$. Hence $J^2 = aJ$, so that $J \in \calX_R$. Similarly for the proof of Example \ref{3.2a}, we have the last assertions.

(5) Follows from the assertions (2), (3), and (4).   
\end{proof}


Closing this section, since the ring as in Example \ref{difficult} is obtained from the gluing of the numerical semigroup $\left<3, 4, 5\right>$, let us explore the $2$-$\AGL$ rings arising as gluing of numerical semigroup rings. 

In what follows, let $0 < a_1, a_2, \ldots, a_{\ell} \in \Bbb Z~(\ell >0)$ be positive integers such that $\GCD(a_1, a_2, \ldots, a_{\ell}) =1$. We set $H_1 =\left<a_1, a_2, \ldots, a_{\ell}\right>$  and assume that $a_1, a_2, \ldots, a_{\ell}$ forms a minimal system of generators of $H_1$. Let $0<\alpha \in H_1$ be an odd integer such that $\alpha \ne a_i$ for every $1 \le i \le \ell$. We consider $H = \left< 2a_1, 2a_2, \ldots, 2a_{\ell}, \alpha \right>$ the gluing of $H_1$ and the set of non-negative integers $\Bbb N$. The reader is referred to \cite[Chapter 8]{RG} for basic properties of gluing of numerical semigroups.
Let $V=k[[t]]$ be the formal power series ring over a field $k$ and set $R_1 = k[[H_1]]$, $R=k[[H]]$ the semigroup rings of $H_1$ and $H$, respectively. We denote by $\m_1$ (resp. $\m$) the maximal ideal of $R_1$ (resp. $R$). Notice that $\mu_R(\m) = \ell + 1$ and $R$ is a free $R_1$-module of rank $2$. By letting ${\rm PF}(H_1) =\{p_1, p_2, \ldots, p_r\}$, the canonical fractional ideal $K_1$ of $R_1$ has the form $K_1 = \sum_{i=1}^r R_1{\cdot}t^{p_r-p_i}$, while $K = \sum_{i=1}^r R{\cdot}t^{2(p_r-p_i)}$ is the canonical fractional ideal of $R$, where $r = \rmr(R_1)$ and $p_r = f(H_1)$. We then have $R \otimes_{R_1} K_1 \cong K$ and hence $K/R \cong R \otimes_{R_1}(K_1/R_1)$ as an $R$-module. We set $\fkc = R:R[K]$. 

With this notation we have the following.

\begin{prop}\label{4.13} 
Suppose that $R_1$ is an $\AGL$ ring, but not a Gorenstein ring. Then the following assertions hold true. 
\begin{enumerate}[$(1)$]
\item $R$ is a $2$-$\AGL$ ring, $\fkc = \m_1 R$, and $\mu_R(\fkc) = \ell \ge 3$.
\item $\fkc \in \calX_R$ if and only if $R_1$ has minimal multiplicity.
\item  $R$ doesn't have minimal multiplicity. Therefore, $\m \notin \calX_R$.
\end{enumerate}
\end{prop}

\begin{proof}
(1) Since $R$ is a free $R_1$-module of rank $2$ and $\ell_R(R/\m_1 R) =2$, we conclude that $R$ is a $2$-$\AGL$ ring (\cite[Theorem 3.10]{CGKM}). Besides, we have $\fkc = \operatorname{Ann}_RK/R = (\operatorname{Ann}_{R_1}{K_1}/{R_1})R = \m_1 R$, whence $\mu_R(\fkc) = \ell \ge 3$.

(2) The isomorphisms $\fkc/\fkc^2 \cong R \otimes_{R_1}(\m_1/{{\m_1}^2}) \cong R\otimes_{R_1}(R_1/\m_1)^{\oplus \ell} \cong (R/\fkc)^{\oplus \ell}$ show that $\fkc/\fkc^2$ is a free $R/\fkc$-module. Hence, $\fkc \in \calX_R$ if and only if $\fkc^2 = f \fkc$ for some $f \in \fkc$. The latter condition is equivalent to saying that $\fkc^2 = t^{2a_i}\fkc$ for some $1 \le i \le \ell$, that is ${\m_1}^2 = t^{2a_i}\m_1$, as desired.

(3) We notice that $\mu_R(\m) = \ell + 1$ and $\rme(R) =\min\{2a_1, 2a_2, \ldots, 2a_{\ell}, \alpha\}$. Suppose that $\rme(R) = 2a_i$ for some $1 \le i \le \ell$. Since $\ell = \mu_{R_1}(\m_1) \le \rme(R_1) \le a_1$, we get 
$$
\rme(R) - \mu_R(\m) = 2a_i - (\ell + 1) \ge 2 \ell -  (\ell + 1) = \ell -1 \ge 2
$$
which implies that $R$ doesn't have minimal multiplicity. Thereafter, we consider the case where $\rme(R) = \alpha$. Suppose that $R$ has minimal multiplicity, that is $\rme(R) = \mu_R(\m)$, in order to seek a contradiction. Since $\alpha$ is an odd integer, we notice that $\ell$ is even, because $\alpha = \rme(R) = \mu_R(\m) = \ell + 1$. Besides, $\alpha < 2 a_i$ for each $1 \le i \le \ell$. Let us write $\alpha = \alpha_1 a_1 + \alpha_2 a_2 + \cdots + \alpha_{\ell}a_{\ell}$ where $\alpha_i \ge 0$. Then one of the $\{\alpha_i\}_{1 \le i \le \ell}$ is positive. Therefore, $\alpha = \alpha_ia_i$ for some $1 \le i \le \ell$, so that $\alpha_i = 1$ and $\alpha = a_i$. This makes a contradiction. Hence $R$ doesn't have minimal multiplicity.
\end{proof}

Consequently, we have the following.

\begin{thm}\label{4.14}
Suppose that $R_1$ is an $\AGL$ ring, but not a Gorenstein ring. Then the following assertions hold true. 
\begin{enumerate}[$(1)$]
\item Let $I \in \calX_R$. Then either $\mu_R(I) = 2$ or $I = \fkc$. 
\item The set of two-generated Ulrich ideals which are generated by monomials in $t$ is 
$$
\left\{ (t^{2m}, t^{\alpha}) \mid 0<m \in H_1,\ \alpha - m \in H_1,\ 2(\alpha - 2m) \in H\right\}.
$$
\end{enumerate}
\end{thm}

\begin{proof}
(1) Thanks to Proposition \ref{Ulrich} (2), if $\mu_R(I) \ge 3$, then $\fkc \subseteq I$. Since $R$ is a $2$-$\AGL$ ring and $\m \notin \calX_R$, we conclude that $I = \fkc$. 

(2) Let $I \in \calX_R$ such that $\mu_R(I) =2$ and $I$ is generated by monomials in $t$. We write $I = (t^p, t^q)$ where $0 < p < q$ and $p, q \in H$. Notice that, for each $0 < h \in H$ with $h \ne \alpha$, we have that $t^h \in \fkc$. Since $I + \fkc = \fkm$ by Theorem \ref{2.3} (2), we get $I \not\subseteq \fkc$, which yields that  $p=\alpha$ or $q = \alpha$. 
The isomorphism $R/I \cong I/(t^p)$ ensures that $p$ is even,  so that $\alpha = q$. Therefore $0 < p < \alpha$. Let us write $p = \sum_{i=1}^{\ell}\left(2a_i\right)c_i + c\alpha = 2\left(\sum_{i=1}^{\ell}a_ic_i\right) + c\alpha$ where $c_i, c \ge 0$. As $p < \alpha$, we have $c = 0$. Therefore, $p= 2m$ for some $0 < m \in H_1$. Moreover, because $I^2 = t^{2m}I$, we have $2(\alpha-2m) \in H$, but $\alpha-2m \notin H$. Since $R/I = R/(t^{2m}, t^{\alpha}) \cong R_1/(t^m, t^{\alpha})$ and $\ell_R(R/I) = m$, we obtain that $t^{\alpha} \in t^m R_1$. Hence $\alpha - m \in H_1$. 

Conversely, let $I = (t^{2m}, t^{\alpha})$ where $0< m_1 \in H_1$, $\alpha - m \in H_1$, and $2(\alpha - 2m) \in H$. We then have $I^2 = t^{2m}I + (t^{2\alpha}) = t^{2m}I$, while $R/I \cong R_1/(t^m, t^{\alpha}) = R_1/t^{m}R_1$, so that $\ell_R(R/I) = m$. Therefore $I \in \calX_R$, as desired.
\end{proof}

\begin{ex}\label{4.16}
Let $H_1 = \left<4, 7, 9\right>$ and $\alpha \ge 11$ an odd integer. We set $R_1=k[[t^4, t^7, t^9]]$ the numerical semigroup ring of $H_1$ over a field $k$. By Example \ref{3.5}, $R_1$ is an $\AGL$ ring with $\rmr(R_1) =2$. Let $H = \left<8, 14, 18, \alpha\right>$ and set $R=k[[H]]$. Then $\mu_R(I) = 2$ for each $I \in \calX_R$.  Moreover, we have the following.
\begin{enumerate}[$(1)$]
\item If $\alpha = 11, 13$, then $\calX_R = \emptyset$.
\item If $\alpha \ge 15$, then $(t^8, t^{\alpha}) \in \calX_R$. 
\item If $\alpha =15$ and $\ch k = 2$, then $(t^8 + ct^{14}, t^{\alpha}) \in \calX_R$ for every $c \in k$, and we have $(t^8 + c_1t^{14}, t^{\alpha}) = (t^8 + c_1t^{14}, t^{\alpha})$, only if $c_1 =c_2$.
\item If $\alpha \ge 17$, then $(t^8 + ct^{14}, t^{\alpha}) \in \calX_R$ for every $c \in k$, and we have $(t^8 + c_1t^{14}, t^{\alpha}) = (t^8 + c_1t^{14}, t^{\alpha})$, only if $c_1 =c_2$.
\end{enumerate}
\end{ex}


\section{G-regularity in $2$-${\rm AGL}$ rings}

The condition that $K/R$ is a free $R/\fkc$-module gives an agreeable restriction on the behavior of $2$-AGL rings, as we have shown in Proposition \ref{1.6} (see also \cite[Section 5]{CGKM}). However, even though $K/R$ is not $R/\fkc$-free, $2$-AGL rings also enjoy nice properties. We will show in the following  that every $2$-AGL ring $R$ is G{\it -regular} in the sense of \cite{greg}, namely, totally reflexive $R$-modules are all free, provided $K/R$ is not $R/\fkc$-free.

\begin{thm}\label{1.5}
Suppose that $R$ is a $2$-${\rm AGL}$ ring, possessing a canonical fractional ideal $K$. We set $\fkc = R : R[K]$, and assume that $K/R$ is not a free $R/\fkc$-module. Let $M$ be a finitely generated $R$-module. If $\Ext_R^p(M, R) =(0)$ for all $p \gg 0$, then $\pd_RM < \infty$. Hence $R$ is $G$-regular in the sense of \cite{greg}.
\end{thm}

\begin{proof}
Let $L=\Omega_R^1(M)$ be the first syzygy module of $M$. For every $p \ge 2$ we have 
$
\Ext_R^{p-1}(L, R) \cong \Ext_R^p(M, R),
$
which shows
$
\Ext_R^p(L, K/R) = (0)$ for all $p \gg 0$,
because  $\Ext_R^p(L,K) = (0)$. Therefore, since $R/\m$ is a direct summand of $K/R$ (Proposition \ref{2.3a} (4)), $\Ext_R^p(L, R/\m) =(0)$ for $p \gg 0$, so that $\pd_RL < \infty$. Hence  $\pd_R M < \infty$. 
\end{proof}

We should compare the following result with \cite[Theorem 2.14 (1)]{GTT2}, where a corresponding result for one-dimensional AGL rings is given.

\begin{cor}\label{4.3}
Suppose that $(R,\m)$ is a $2$-${\rm AGL}$ ring with minimal multiplicity, possessing a canonical fractional ideal $K$ and $\fkc = R: R[K]$. Then 
$$
\calX_R= \begin{cases}
\{\fkc, \m\}, &  \ \text{if} \ K/R~\text{is}~R/\fkc\text{-free},\\
\{\m\}, & \ \text{otherwise}.
\end{cases}
$$
\end{cor}

\begin{proof}
Since $R$ has minimal multiplicity, $\m \in \calX_R$, so that $\calX_R \ne \emptyset$.

$(1)$ Suppose that $K/R$ is $R/\fkc$-free. Then, by \cite[Proposition 5.7 (1)]{CGKM}, $\m:\m$ is a local ring, while $S=R[K]$ is a Gorenstein ring, since $R$ is a $2$-AGL ring with minimal multiplicity  (\cite[Corollary 5.3]{CGKM}). Therefore, thanks to Proposition \ref{1.6}, $\fkc =R:S \in \calX_R$, so that $\{\fkc, \m\} \subseteq \calX_R$. Let $I \in \calX_R$. Then, because $R$ has minimal multiplicity, $\mu_R(I) \ge 3$ by Corollary \ref{1.7}. Therefore, since  $K/R$ is $R/\fkc$-free, we get $\fkc = (0):_RK/R \subseteq I$ (\cite[Corollary 2.13]{GTT2}). Thus, $I = \fkc$ or $I = \m$, because $\ell_R(R/\fkc)=2$.

$(2)$ Suppose that $K/R$ is not $R/\fkc$-free and let $I$ be an Ulrich ideal of $R$. Then, $\mu_R(I) \ge 3$ by Theorem \ref{2.3}. Therefore, thanks to the proof of case (1), $I=\fkc$ or $I=\m$. Thus, $I = \fkm$, because $\fkc \not\in \calX_R$ by Proposition \ref{1.6}. 
\end{proof}

We close this paper with the following, where two kinds of $2$-${\rm AGL}$ rings of minimal multiplicity are given, one is $R/\fkc$-free and the other one is not.

\begin{ex}
Let $V=k[[t]]$ denote the formal power series ring over a field $k$ and set $R_1 =k[[t^3, t^7, t^8]]$, $R_2=k[[t^4, t^9, t^{11}, t^{14}]]$. Let $K_i$ be a canonical fractional ideal of $R_i$. Then, both $R_1$ and $R_2$ are $2$-AGL rings. We have $K_1/R_1$ is a free $R/\fkc_1$-module, but $K_2/R_2$ is not $R/\fkc_2$-free, where $\fkc_i = R_i : R_i[K_i]$. Therefore,  $\calX_{R_1}=\{(t^6, t^7, t^8), (t^3, t^7, t^8)\}$, and $\calX_{R_2}=\{(t^4, t^9, t^{11}, t^{14})\}$.
\end{ex}

\begin{proof}
We have $K_1 = R + Rt$ and $K_2 = R + Rt + Rt^5$. Hence, $R_1[K_1] = R[t] = V$, and $R_2[K_2] = R[t^3, t^5] = k[[t^3, t^4, t^5]]$, so that 
$
\ell_{R_1}({R_1[K_1]}/{K_1})  = \ell_{R_2}({R_2[K_2]}/{K_2}) = 2.
$
Therefore, by Theorem \ref{mainref}, both $R_1$ and $R_2$ are $2$-AGL rings. Because 
$
\ell_{R_1}(K_1/{R_1}) = 2 \ \ \text{and} \ \  \ell_{R_2}(K_2]/{R_2}) = 3,
$
$K_1/{R_1}$ is a free $R/\fkc_1$-module, but $K_2/R_2$ is not $R/\fkc_2$-free (use Proposition \ref{2.3a} (4)). Notice that $R_1$ and $R_2$ have minimal multiplicity $3$ and $4$, respectively. Hence, the results readily follow from Corollary \ref{4.3}, since $\fkc_1 =R_1:V= t^6 V = (t^6, t^7, t^8)$.
\end{proof}


\section{Appendix: Ulrich ideals in one-dimensional Gorenstein local rings of finite Cohen-Macaulay representation type}

In \cite{GOTWY}, the authors determined all the Ulrich ideals in one-dimensional Gorenstein local rings $R$ of finite CM-representation type, while in \cite[Section 12]{GTT} most birational module-finite extensions of these rings have been searched. Since the proof given by \cite{GOTWY} depends on the techniques in the representation theory of maximal Cohen-Macaulay modules, it might have some interests to give a straightforward proof, making use of the results of \cite[Section 12]{GTT} and determining the members of $\calA_R^0$  by Lemma \ref{lem3.1}, as well. We note it  as an appendix.

In this appendix, let $(R, \m)$ be a Gorenstein complete local ring of dimension one with algebraically closed residue class field $k$ of characteristic $0$. Suppose that $R$ has finite CM-representation type. Then, by \cite[(8.5), (8.10), and (8.15)]{Y} we get   
$$
R \cong k[[X, Y]]/(f),
$$
where $k[[X, Y]]$ is the formal power series ring over $k$, and $f$ is one of the following polynomials.
\begin{itemize}
\item[$(A_n)$] $X^2-Y^{n+1}$ $(n \ge 1)$
\item[$(D_n)$] $X^2Y-Y^{n-1}$ $(n \ge 4)$
\item[$(E_6)$] $X^3-Y^4$
\item[$(E_7)$] $X^3-XY^3$
\item[$(E_8)$] $X^3-Y^5$
\end{itemize}

With this notation we have the following.

\begin{thm}[{\cite[Theorem 1.7]{GOTWY}}]\label{5.1}
The set $\calX_R$ is given by the following.
\begin{enumerate}
\item[$(A_n)$] $\calX_R = 
\begin{cases}
\left\{(x, y^q) \mid 0 < q \le \ell \right\} & \text{if} \  n=2 \ell -1 \ \text{with} \  \ell \ge 1,\\
\left\{(x, y^q) \mid 0 < q \le \ell \right\} & \text{if} \ n=2\ell \ \text{with} \ \ell \ge 1.
\end{cases}$
\item[$(D_n)$] $\calX_R = 
\begin{cases}
\left\{(x^2, y), (x, y^{\ell+1})  \right\} & \text{if} \  n=2 \ell +3 \ \text{with} \ \ell \ge 1,\\
\left\{(x^2, y), (x-y^{\ell}, y(x+y^{\ell})), (x+y^{\ell}, y(x-y^{\ell})) \right\} & \text{if} \ n=2(\ell + 1) \ \text{with} \  \ell \ge 1.
\end{cases}$
\item[$(E_6)$] $\calX_R= \left\{(x, y^2)\right\}$
\item[$(E_7)$] $\calX_R = \left\{(x,y^3)\right\}$
\item[$(E_8)$] $\calX_R = \emptyset$
\end{enumerate}
where $x$ and $y$ denote the images of $X$ and $Y$ in the corresponding rings, respectively.
\end{thm}

\begin{proof} For a ring $A$, let $J(A)$ denote its Jacobson radical. We denote by $\overline{R}$ the integral closure of $R$ in $\rmQ(R)$, and by $\calB_R$ the set of birational module-finite extensions of $R$.

{\bf (1)} ($E_6$) See Example \ref{3.2a}.


{\bf (2)} ($E_8$) Let $R=k[[t^3,t^5]]$ and $V=k[[t]]$. By \cite[Proposition 12.7 (3)]{GTT}, $\calB_R= \{R, k[[t^3,t^5, t^7]], k[[t^3,t^4,t^5]], k[[t^2,t^3]], V\}$, among which $k[[t^3,t^5, t^7]], k[[t^3,t^4,t^5]]$ are not Gorenstein rings, and $\mu_R(V)=\mu_R(k[[t^2,t^3]])=3$. Hence, $\calA_R^0=\emptyset$, so that $\calX_R= \emptyset$ by Lemma \ref{lem3.1}.


{\bf (3)} ($E_7$) Let $R=k[[X,Y]]/(X^3-XY^3)$. We set $S=k[[X, Y]]$, $V=k[[t]]$, and $f=X^3-XY^3$. Then, since $(f)=(X)\cap (X^2-Y^3)$, we get the tower 
$$
R=S/(f) \subseteq S/(X)\oplus S/(X^2-Y^3) = k[[Y]] \oplus k[[t^2, t^3]] \subseteq k[[Y]] \oplus V = \overline{R}
$$
of rings, where we identify $S/(X) = k[[Y]]$ and $S/(X^2-Y^3) = k[[t^2, t^3]] \subseteq V$.

\begin{claim}\label{5.2}
$
\calA_R = \{ R, k[[Y]] \oplus k[[t^2, t^3]], k[[Y]] \oplus V, k + J(\overline{R}) \}.
$
\end{claim}

\begin{proof}[Proof of Claim \ref{5.2}]
Let $A \in \calB_R$ such that $R \ne A$ and let $p_2 : \overline{R} \to V$ denote the projection. We set $B=p_2(A)$. Since $k[[t^2, t^3]] \subseteq B \subseteq V$, $B= k[[t^2, t^3]]$ or $B=V$. Suppose that $A$ is not a local ring. Then, $A$ decomposes into a direct product of local rings, since $A$ is a module-finite extension of the complete local ring $R$, so that  we may choose a non-trivial idempotent $e \in A$. Then, since $\overline{R} = k[[X]] \oplus V$, we get $e= (1,0)$, or $(0,1)$, whence $(1,0), (0,1) \in A$, so that $A=A(1,0) + A(0,1) = k[[Y]] \oplus B$. Suppose that $A$ is a local ring. In this case, the argument in \cite[Pages 2708--2710]{GTT} shows that if $B=V$, then 
$A \cong k[[Y, Z]]/(Z(Y-Z^2)) = k[[(Y, t^2), (0, t)]] =k + J(\overline{R})$, and that if $B=k[[t^2, t^3]]$, then $A$ is an AGL but  not a Gorenstein ring. Thus we have the assertion. 
\end{proof}

Since $J(\overline{R})=R(Y,t^2)+R(0,t)+R(0,t^2)$, we have $k+J(\overline{R})= R + R(0,t)+R(0,t^2)$, whence $\mu_R(k+J(\overline{R})) = 3$. Therefore, $\calA_R^0 = \left\{k[[Y]] \oplus k[[t^2, t^3]]\right\}$, so that  by Lemma \ref{lem3.1} $\calX_R=\{(x,y^3)\}$, since $R:(k[[Y]] \oplus k[[t^2, t^3]])=(x,y^3)$.

\medskip
{\bf (4)} ($D_n$) {\em $(\rm i)$ $($The case where $n=2 \ell + 3$ with $\ell \ge 1$$)$.}
Let $R=k[[X,Y]]/(X^2Y-Y^{2\ell +2})$. We set $S=k[[X, Y]]$, $V=k[[t]]$, and $f= Y(X^2-Y^{2\ell + 1})$. We consider the tower
$$
R=S/(f) \subseteq S/(Y)\oplus S/(X^2-Y^{2\ell + 1}) = k[[X]] \oplus k[[t^2, t^{2\ell + 1}]] \subseteq k[[X]] \oplus V = \overline{R}
$$
of rings, where we identify $S/(Y) = k[[X]]$ and $S/(X^2-Y^{2\ell + 1}) = k[[t^2, t^{2\ell + 1}]]$. We then have the following.

\begin{claim}\label{5.3}
$\calA_R = \left\{R, k+J(\overline{R})\right\} \cup \left\{k[[X]] \oplus k[[t^2, t^{2q+1}]] \mid 0 \le q \le \ell \right\}$
\end{claim}

\begin{proof}[Proof of Claim \ref{5.3}]
Let $A \in \calB_R$ such that $R \ne A$ and let $p_2 : \overline{R} \to V$ denote the projection. We set $B=p_2(A)$. Then, by \cite[Corollary 12.5 (1)]{GTT} $B= k[[t^2, t^{2q+1}]]$ for some $0 \le q \le \ell$, since $k[[t^2, t^{2\ell +1}]] \subseteq B \subseteq V$. If $A$ is not a local ring, then the same proof as in Claim \ref{5.2} works, to get $A= k[[X]] \oplus B$. If $A$ is a local ring, then by the argument in \cite[Pages 2710--2711]{GTT} we have 
$
A \cong k[[X, Z]]/[(Z)\cap(X-Z^{2\ell + 1})] = k + J(\overline{R})$.
\end{proof}
Consequently, $\calA_R^0=\left\{k[[X]]\oplus k[[t^2,t^{2\ell +1}]], k+J(\overline{R})\right\}$. We have $$\left(k[[X]]\oplus k[[t^2,t^{2\ell +1}]]\right)/R \cong S/(X^2, Y)$$ and $k+J(\overline{R})=R+R(0,t)$. Therefore, Lemma \ref{lem3.1} shows the assertion, because $$R:\left(k[[X]]\oplus k[[t^2,t^{2\ell +1}]]\right) = (x^2, y) \ \ \text{and}\ \ R:\left(k+J(\overline{R})\right)=(x, y^{\ell+1}).$$

\medskip
{\bf (4)} ($D_n$) {\em $(\rm ii)$ $($The case where $n=2 \ell + 2$ with $\ell \ge 1$$)$.}
Let $R=k[[X,Y]]/(X^2Y-Y^{2\ell +1})$. We set $S=k[[X, Y]]$, $V=k[[t]]$, and $f=Y(X^2-Y^{2\ell})$. Consider the tower
$$
R=S/(f) \subseteq S/(Y)\oplus T = k[[X]] \oplus \overline{T} = \overline{R}
$$
of rings, where $T=S/(X^2-Y^{2\ell})$. By \cite[Page 2711]{GTT} an  intermediate ring $R \subsetneq A \subseteq \overline{R}$ is an AGL ring but not a Gorenstein ring, if $A$ is a local ring. Therefore, every $A \in \calA_R$ is not local, if $R\ne A$.

\begin{claim}\label{5.5}
$\calA_R = \left\{R, S/(X-Y^{\ell})\oplus S/(Y(X+Y^{\ell})), S/(X+Y^{\ell})\oplus S/(Y(X-Y^{\ell})\right)\} \cup \left\{k[[X]] \oplus T[\frac{x}{y^q}] \mid 0 \le q \le \ell \right\}$
\end{claim}

\begin{proof}[Proof of Claim \ref{5.5}]
Let $A \in \calA_R$ such that $R \ne A$. Note that $\overline{R}=k[[X]] \oplus S/(X-Y^{\ell}) \oplus S/(X+Y^{\ell})$. Let $\{e_i\}_{i=1,2,3}$ be the orthogonal primitive idempotents of $\overline{R}$. Then, $e_i \in A$ for some $1 \le i \le 3$, since $A$ is not a local ring. If $A \ne \overline{R}$, such $e_i$ is unique for $A$.

{\em $(\rm i)$ $($The case where $e_1 \in A$$)$.} Let $p:\overline{R} \to S/(X-Y^{\ell}) \oplus S/(X+Y^{\ell})$ denote the projection. Then  
$$
T:=S/{(X-Y^{\ell}) \cap (X+Y^{\ell})} \subseteq p(A) \subseteq \overline{T} = S/(X-Y^{\ell}) \oplus S/(X+Y^{\ell})
$$
so that, by \cite[Corollary 12.5 (2)]{GTT} $p(A) = T[\frac{x}{y^q}]$ for some $0 \le q < \ell$. Hence, $A = k[[X]] \oplus T[\frac{x}{y^q}]$.

{\em $(\rm ii)$ $($The case where $e_2 \in A$$)$.} Let $p:\overline{R} \to k[[X]] \oplus S/(X+Y^{\ell})$ denote the projection. Because $A \ne \overline{R}$, we have
$$
S/{(Y) \cap (X+Y^{\ell})} \subseteq p(A) \subsetneq k[[X]] \oplus S/(X+Y^{\ell}),
$$
which shows $p(A) = S/{(Y) \cap (X+Y^{\ell})} = S/(Y(X+Y^{\ell}))$. Thus, $A = S/(X-Y^{\ell}) \oplus S/(Y(X+Y^{\ell}))$. Similarly, $A=S/(X+Y^{\ell})\oplus S/(Y(X-Y^{\ell}))$ if $e_3 \in A$, which proves Claim \ref{5.5}.
\end{proof}

Therefore, $$\calA_R^0 =\left\{k[[X]] \oplus T, S/(X-Y^{\ell})\oplus S/(Y(X+Y^{\ell})), S/(X+Y^{\ell})\oplus S/(Y(X-Y^{\ell}))\right\},$$ so that $\calX_R= \left\{(x^2, y),  (x-y^{\ell}, y(x+y^{\ell})), (x+y^{\ell}, y(x-y^{\ell})) \right\}$.

\medskip
{\bf (5)} ($A_n$) {\em $(\rm i)$ $($The case where $n=2 \ell$ with $\ell \ge 1$$)$.}
Let $R=k[[t^2, t^{2\ell +1}]]$. Then, $\calA_R^0 = \{k[[t^2, t^{2q+1} \mid 0 \le q \le \ell -1]]\}$ by \cite[Corollary 12.5 (1)]{GTT}, whence $\calX_R =\{ (x, y^q) \mid 0 < q \le \ell \}$.

\medskip
{\bf (5)} ($A_n$) {\em $(\rm ii)$ $($The case where $n=2 \ell - 1$ with $\ell \ge 1$$)$.}
Let $R=k[[X,Y]]/(X^2-Y^{2\ell})$. We set $S=k[[X, Y]]$ and  $f=X^2-Y^{2\ell}=(X-Y^{\ell})(X+Y^{\ell})$. We then have $\ell_R(\overline{R}/R)=\ell$ by the exact sequence
$$
0 \to R=S/(f) \longrightarrow \overline{R}=S/(X-Y^\ell)\oplus S/(X+Y^\ell) \longrightarrow S/(X, Y^\ell) \to 0
$$
of $R$-modules.  Let $A \in \calA_R$ such that $R \ne A$. Then, by \cite[Corollary 12.5 (2)]{GTT} 
$
A = R\left[\frac{x}{y^q}\right]  \ \ \text{for some} \ \  0 < q \le \ell
$
in $\rmQ(R)$. If $n=\ell$, then $A = \overline{R}$ is a Gorenstein ring with $\mu_R(\overline{R})=2$, so that $(x, y^\ell)= R:\overline{R} \in \calX_R$.

Let us now assume that $0 < q < \ell$. Since $(\frac{x}{y^q})^2 =x^2y^{-2q} = y^{2\ell}y^{-2q}=y^{2(\ell - q)}\in R$, we have 
$
A=R+R{\cdot}\frac{x}{y^q}.
$
We will show that $A$ is a Gorenstein local ring with $\mu_R(A) =2$. Indeed, set let $\n = \m A+ \frac{x}{y^q}A$ of $A$, and let $M$ be an arbitrary maximal ideal of $A$. We choose a maximal ideal $N$ of $\overline{R}$ so that $M = N \cap A$. We then have 
$
N \supseteq J(\overline{R}) \supseteq y\overline{R} + \frac{x}{y^q}\overline{R},
$
whence $M=N\cap A \supseteq \n$, so that $M=\n$ because $\n$ is a maximal ideal of $A$. Hence, $(A,\n)$ is a local ring. Consequently,  $2 \le \mu_R(A) = \ell_R(A/\m A) \le \e(A) \le \e(R)=2$. Thus $A \in \calX_R^0$. Note that $R:A = R:_R\frac{x}{y^q}$, because $A=R+R\frac{x}{y^q}$. We now take $a \in R:\frac{x}{y^n}$. Then, setting $b=a{\cdot}\frac{x}{y^q} \in R$, we have $ax=by^q$, so that $AX-BY^q=C(X^2-Y^{2\ell})$ for some $C \in S$. Here $a, b$ are the images of $A, B$ respectively. Therefore 
$
X(A-CX)=Y^q(B-Y^{2\ell-q}).
$
Since $X, Y^q$ forms an $S$-regular sequence, we have
$
A-CX=Y^qD \text{ for some }D \in S.
$
Hence, $a \in (x,y^q)R$, so that $R:A=(x,y^q)$. Therefore $\calX_R = \{(x, y^q) \mid 0 < q \le \ell\}$. 
\end{proof}

\begin{rem}
The assertion on the ring of type $(A_n)$ also follows from \cite[Theorem 4.5]{GIK}. In fact, the ring $R$ of type $(A_n)$ has minimal multiplicity $2$. Hence, by \cite[Theorem 4.3]{GIK} $\calX_R$ is totally ordered with respect to inclusion, and $R:\overline{R}$ is the minimal element of $\calX_R$. 
\end{rem}



\end{document}